\documentclass[reqno]{amsart}
\usepackage[english]{babel}
\usepackage{amssymb,amsmath,hyperref}
\usepackage{bbold}
\usepackage{amsrefs}
\usepackage{stackrel}
\usepackage[foot]{amsaddr}
\usepackage{mathrsfs}
\usepackage{cleveref, enumerate}

\usepackage{transfer}

\newtheorem{thm}{Theorem}

\newtheorem{cor}[thm]{Corollary}
\newtheorem{defi}[thm]{Definition}
\newtheorem{rem}[thm]{Remark}
\newtheorem{nota}[thm]{Notation}
\newtheorem{exa}[thm]{Example}

\newtheorem{ack}[thm]{Acknowledgement}

\newtheorem*{tempo*}{Template}

\newcommand\be{\begin{equation}}
\newcommand\ee{\end{equation}}

\newbox\gnBoxA
\newdimen\gnCornerHgt
\setbox\gnBoxA=\hbox{$\ulcorner$}
\global\gnCornerHgt=\ht\gnBoxA
\newdimen\gnArgHgt

\def\Godelnum #1{%
	\setbox\gnBoxA=\hbox{$#1$}%
	\gnArgHgt=\ht\gnBoxA%
	\ifnum \gnArgHgt<\gnCornerHgt
		\gnArgHgt=0pt%
	\else
		\advance \gnArgHgt by -\gnCornerHgt%
	\fi
	\raise\gnArgHgt\hbox{$\ulcorner$} \box\gnBoxA %
		\raise\gnArgHgt\hbox{$\urcorner$}}
\def\bdefi{\begin{defi}\rm}
\def\edefi{\end{defi}}
\def\bnota{\begin{nota}\rm}
\def\enota{\end{nota}}
\def\brem{\begin{rem}\rm}
\def\erem{\end{rem}}

\def\RCA{\textup{\textsf{RCA}}}

\def\RCAO{\textup{\textsf{RCA}}_{0}^{\Omega}}

\def\ef{\textup{\textsf{ef}}}
\def\ns{\textup{\textsf{ns}}}

\def\WKL{\textup{\textsf{WKL}}}

\def\WWKL{\textup{\textsf{WWKL}}}

\def\T{\mathcal{T}}

\def\bye{\end{document}}

\def\N{{\mathbb  N}}

\def\R{{\mathbb  R}}

\def\WEI{\textup{\textsf{WEI}}}

\def\MUC{\textup{\textsf{MUC}}}

\def\ULC{\textup{\textsf{ULC}}}
\def\ULD{\textup{\textsf{ULD}}}

\def\R{{\mathbb{R}}}
\def\b{{\mathfrak{b}}}
\def\({\textup{(}}
\def\){\textup{)}}

\def\st{\textup{st}}

\def\nst{\textup{nst}}
\def\asa{\leftrightarrow}

\def\di{\rightarrow}

\def\eps{\varepsilon}
\def\epsa{\varepsilon_{\alpha}}

\def\M{\textsf{\textup{M}}}

\def\paai{\Pi_{1}^{0}\textup{-\textsf{TRANS}}}

\def\QFAC{\textup{\textsf{QF-AC}}}

\def\r{{\mathfrak{r}}}
\def\PRA{\textup{\textsf{PRA}}}
\def\app{\textup{\textsf{app}}}
\def\EPA{\textup{\textsf{E-PA}}}

\def\MU{\textup{\textsf{MU}}}

\def\MAC{\textup{\textsf{mAC}}_{\st}^{\omega}}
\def\FTC{\textup{\textsf{FTC}}}

\def\CRI{\textup{\textsf{CRI}}}

\numberwithin{equation}{section}
\numberwithin{thm}{section}

\usepackage{comment}

\begin{document}
\title[The computational content of standard mathematics]{Non-standard Nonstandard Analysis and the computational content of standard mathematics}
\author{Sam Sanders}
\address{Munich Center for Mathematical Philosophy, LMU Munich, Germany \& Department of Mathematics (S22), Krijgslaan 281, Ghent University, 9000 Ghent, Belgium}
\email{sasander@me.com}
\begin{abstract}
The aim of this paper is to highlight a hitherto unknown \emph{computational} aspect of Nonstandard Analysis.  
Recently, a number of nonstandard versions of G\"odel's system $\textsf{T}$ have been introduced (\cites{brie, fega,benno2 }), and it was 
shown in \cite{sambon} that the systems from \cite{brie} play a pivotal role in extracting computational information from \emph{proofs in Nonstandard Analysis}.  
It is a natural question if similar techniques may be used to extract computational information from proofs \emph{not involving Nonstandard Analysis}.  
In this paper, we provide a positive answer to this question using the nonstandard system from \cite{fega}.  
This system validates so-called non-standard uniform boundedness principles which are central to Kohlenbach's approach to proof mining (\cite{kohlenbach3}).    
In particular, we show that from \emph{classical and ineffective} existence proofs (not involving Nonstandard Analysis but using weak K\"onig's lemma), one can `automatically' extract approximations to the objects claimed to exist. 
\end{abstract}

\maketitle
\thispagestyle{empty}

%

\section{Introduction}\label{intro}
The aim of this paper is to highlight a hitherto unknown \emph{computational} aspect of Nonstandard Analysis.  
In two words, we show how the nonstandard system introduced in \cite{fega} can be used to extract computational information from proofs \emph{not involving Nonstandard Analysis}.  
We provide more details in the next section.  
\subsection{A new computational aspect of Nonstandard Analysis}\label{wintro}
Recently, a number of nonstandard versions of G\"odel's system $\textsf{T}$ have been introduced (\cites{brie, fega,benno2}), based on Nelson's \emph{internal set theory} (\cite{wownelly}).   
Using the systems from \cite{brie}, an algorithm was formulated in \cite{sambon} which allows for the extraction of computational information from proofs of theorems from `pure' Nonstandard Analysis, i.e.\ formulated solely with \emph{nonstandard} definitions (of continuity, convergence, differentiability, compactness, et cetera).  Note that the results from \cite{sambon} are not needed for this paper, but merely serve as a motivation or starting point.

\medskip

Intuitively speaking, the algorithm from \cite{sambon}*{\S3.5} produces a proof of the `constructive' version of the theorem at hand (if possible), or the associated equivalence from \emph{Reverse Mathematics} (See \cite{simpson2, simpson1} for the latter) otherwise.      
Here, the word `constructive' may be interpreted either in the mainstream sense `effective', or the foundational sense of Bishop's \emph{Constructive Analysis};  Both cases are treated in \cite{sambon}.  
We believe our aims to be best further illustrated by way of an example, as follows.  
\begin{exa}\rm
Assuming basic familiarity with Nelson's internal set theory, we consider the following theorem:  \emph{A uniformly continuous function on the unit interval is Riemann integrable there}.  The previous theorem formulated with the nonstandard 
definitions of continuity and integration is:  
\begin{align}
(\forall f:\R\di \R)\big[(\forall x, y\in [0,1])&[x\approx y \di f(x)\approx f(y)] \label{NST}\\
&\di  (\forall \pi, \pi' \in P([0,1]))(\|\pi\|,\| \pi'\|\approx 0  \di S_{\pi}(f)\approx S_{\pi}(f) )  \big],\notag
\end{align}
where, as suggested by the notation, $\pi$ and $\pi'$ are discrete partitions of the unit interval, and $\|\pi\|$ is the mesh of a partition, i.e.\ the largest distance between two adjacent partition points;  
The number $S_{\pi}(f)$ is the \emph{Riemann sum} associated with the partition $\pi$, i.e.\ $\sum_{i=0}^{M-1}f(t_{i}) (x_{i}-x_{i+1}) $ for $\pi=(0, t_{0}, x_{1},t_{1},  \dots,x_{M-1}, t_{M-1}, 1)$.

\medskip
\noindent
By contrast, the \emph{effective version} of the theorem is as follows:
\begin{align}\textstyle
\textstyle~(\forall f: \R\di \R, g ,n)\Big[(\forall &\textstyle x, y \in [0,1],k)(|x-y|<\frac{1}{g(k)} \di |f(x)-f(y)|\leq\frac{1}{k})\label{EST}\\
&\textstyle\di  (\forall \pi, \pi' \in P([0,1]))\big(\|\pi\|,\| \pi'\|< \frac{1}{t(g,n)}  \di |S_{\pi}(f)- S_{\pi}(f)|\leq \frac{1}{n} \big)  \Big],\notag
\end{align}
where $t$ is a term from the original language, i.e.\ \emph{not involving} Nonstandard Analysis.
In \cite{sambon}*{\S3.1} it is shown how a proof of the nonstandard version of \eqref{NST} can be converted into a proof of the effective version \eqref{EST};  Furthermore, the term $t$ in \eqref{EST} can be `read off' from the proof of \eqref{NST}. 
Finally, an algorithm is formulated in \cite{sambon}*{\S3.5} which converts the proof of a theorem of pure Nonstandard Analysis, like \eqref{NST}, into the proof of associated effective version, like \eqref{EST}, and also extracts the required term $t$.  
\end{exa}
In light of the previous example from \cite{sambon}, we observe it is possible to extract computational information from certain proofs in Nonstandard Analysis.  
It is then a natural question if similar techniques may be used to extracted computational information from proofs \emph{not involving Nonstandard Analysis}.  
In this paper, we provide a positive answer to this question.  The nonstandard system $\M$, introduced by Ferreira-Gaspar in \cite{fega} and discussed in Section~\ref{otherm} below, plays a central role.    
As pointed out in \cite{fega}*{\S1-2}, $\M$ validates so-called non-standard uniform boundedness principles which are central to Kohlenbach's approach to proof mining (\cite{kohlenbach3}).    
Note that Kohlenbach uses the word `non-standard' in \cite{kohlenbach3} to refer to the \emph{non-classical} status of the uniform boundedness principles.

\medskip

As to methodology, we shall establish more `non-standard' features of $\M$ in Section~\ref{NSANSA}, which imply the equivalence between the usual $\eps$-$\delta$-definitions (of continuity, differentiability, convergence, et cetera) and the associated nonstandard definitions.  
Such an equivalence does not follow from the usual base theory or standard principles.  

\medskip

Thanks to the aforementioned equivalences, a proof of a mathematical theorem only involving $\eps$-$\delta$-definitions, gives rise to a proof of the associated \emph{nonstandard} version in $\M$.  
From the latter proof, computational information can be extracted using (substantial modifications of) the techniques from \cite{sambon}.  We treat the fundamental theorem of calculus as an example in Section~\ref{examinisalamini}, 
while further examples are treated in Section~\ref{convieneient} and \ref{papadimi}.    
In Section~\ref{sucktwo}, we provide a general template for extracting computational information using $\M$ from a proof not involving Nonstandard Analysis (but weak K\"onig's lemma is allowed).  
In Section~\ref{LIM}, we establish the limitations and scope of this template.  

\medskip

A particularly nice result is obtained in Section \ref{papprox}, namely that from certain classical existence proofs (not involving Nonstandard Analysis), one can `automatically' extract approximations to the objects claimed to exist. 
If the object claimed to exist is unique, the approximations converge to the former.     
Examples which can be treated in this way are:  The intermediate value theorem, the Weierstra\ss~maximum theorem, the Peano existence theorem, and the Brouwer fixed point theorem (See \cite{simpson2}*{II.6 and IV.2})

\medskip

In conclusion, we shall use a system of `non-standard' Nonstandard Analysis to extract computational information from proofs not involving Nonstandard Analysis.  

\subsection{Internal set theory and its fragments} \label{otherm}
In this section, we discuss Nelson's \emph{internal set theory}, first introduced in \cite{wownelly}, and its fragment $\M$ from \cite{fega}.  

\medskip

In Nelson's \emph{syntactic} approach to Nonstandard Analysis (\cite{wownelly}), as opposed to Robinson's semantic one (\cite{robinson1}), a new predicate `st($x$)', read as `$x$ is standard' is added to the language of \textsf{ZFC}, the usual foundation of mathematics.  
The notations $(\forall^{\st}x)$ and $(\exists^{\st}y)$ are short for $(\forall x)(\st(x)\di \dots)$ and $(\exists y)(\st(y)\wedge \dots)$.  A formula is called \emph{internal} if it does not involve `st', and \emph{external} otherwise.   
The three external axioms \emph{Idealisation}, \emph{Standard Part}, and \emph{Transfer} govern the new predicate `st';  They are respectively defined\footnote{The superscript `fin' in \textsf{(I)} means that $x$ is finite, i.e.\ its number of elements are bounded by a natural number.      } as:  
\begin{enumerate}
\item[\textsf{(I)}] $(\forall^{\st~\textup{fin}}x)(\exists y)(\forall z\in x)\varphi(z,y)\di (\exists y)(\forall^{\st}x)\varphi(x,y)$, for internal $\varphi$.  
\item[\textsf{(S)}] $(\forall x)(\exists^{\st}y)(\forall^{\st}z)\big([z\in x\wedge \varphi(z)]\asa z\in y\big)$, for any $\varphi$.
\item[\textsf{(T)}] $(\forall^{\st}x)\varphi(x, t)\di (\forall x)\varphi(x, t)$, where $\varphi$ is internal,  $t$ captures \emph{all} parameters of $\varphi$, and $t$ is standard.  
\end{enumerate}
The system \textsf{IST} is (the internal system) \textsf{ZFC} extended with the aforementioned three external axioms;  
The former is a conservative extension of \textsf{ZFC} for the internal language, as proved in \cite{wownelly}.    

\medskip

In \cite{fega}, the authors study G\"odel's system $\textsf{T}$ extended with special cases of the external axioms of \textsf{IST}.  
In particular, they consider nonstandard extensions of the (internal) systems \textsf{E-HA}$^{\omega}$ and $\textsf{E-PA}^{\omega}$, respectively \emph{Heyting and Peano arithmetic in all finite types and the axiom of extensionality}.       
We refer to \cite{brie}*{\S2.1} for the exact details of these (mainstream in mathematical logic) systems.  

\medskip

The results in \cite{fega} are inspired by those in \cite{brie}.  In particular, the notion of finiteness central to the latter is replaced by the notion of \emph{strong majorizability}.  The latter notion and the associated system $\M$ is introduced in the next paragraph, assuming familiarity with the higher-type framework of G\"odel's $\textsf{T}$ (See e.g.\ \cite{brie}*{\S2.1}).    

\medskip

The system $\M$, a conservative extension of $\textsf{E-PA}^{\omega}$ (See \cite{fega}*{Cor.\ 1}) is based on the Howard-Bezem notion of \emph{strong majorizability}.  We first introduce the latter and related notions.  
For more extensive background on strong majorizability, we refer to \cite{kohlenbach3}*{\S3.5}.    
\bdefi[Majorizability] The strong majorizability predicate `$\leq^{*}$' is inductively defined as follows:
\begin{itemize}
\item $x \leq_{0}^{*} y$ is $x \leq_{0} y$;
\item $x \leq_{\rho\di \sigma}^{*} y$ is   $(\forall u)(\forall v\leq_{\rho}^{*}u)\big(xu\leq_{\sigma}^{*} yv \wedge yu\leq_{\sigma}^{*} yv  \big)  $.   
\end{itemize}
An object $x^{\rho}$ is called \emph{monotone} if $x\leq_{\rho}^{*}x$.  The quantifiers $(\tilde{\forall} x^{\rho})$ and $(\tilde{\exists} y^{\rho} )$ range over the monotone objects of type $\rho$, i.e.\ they are abbreviations for the formulas $(\forall x)(x\leq^{*}x \di \dots)$ and 
$(\exists y)(y\leq^{*}y\wedge\dots)$.
\edefi
As in \cite{fega}*{\S2}, we have the following definition.  The language of $\textsf{E-PA}_{\st}^{\omega}$ is the language of $\textsf{E-PA}^{\omega}$ extended with a new predicate $\st^{\sigma}$ for every finite type $\sigma$.  
The typing of the standardness predicate is usually omitted.    
\bdefi[Basic axioms]\label{debs}
The system $ \textsf{E-PA}^{\omega}_{\st} $ is defined as $ \textsf{E-PA}^{\omega} + \T^{*}_{\st} + \textsf{IA}^{\st}$, where $\T^{*}_{\st}$
consists of the following axiom schemas.
\begin{enumerate}
\item $x =_{\sigma} y \di (\st^{\sigma}(x) \di \st^{\sigma}(y));$
\item $\st^{\sigma}(y) \di (x\leq_{\sigma}^{*} y \di \st^{\sigma}(x))$;
\item $\st^{\sigma}(t)$, for each closed term $t$ of type $\sigma$; 
\item $\st^{\sigma\di \tau} (z) \di (\st^{\sigma} (x) \di \st^{\tau}(zx))$.
\end{enumerate}
The \emph{external induction axiom} \textsf{IA}$^{\st}$ is the following schema for any $\Phi$: 
\be\tag{\textsf{IA}$^{\st}$}
\Phi(0)\wedge(\forall^{\st}n^{0})(\Phi(n) \di\Phi(n+1))\di(\forall^{\st}n^{0})\Phi(n).     
\ee
\edefi
\bdefi[Non-basic axioms]\label{nba`'}~
\begin{enumerate}
\item Monotone Choice $ \textsf{mAC}_{\st}^{\omega}$:     For any internal $\varphi$, we have   
\[
(\tilde{\forall}^{\st} x)(\tilde{\exists}^{\st}y)\varphi(x, y)\di (\tilde{\exists}^{\st}f)(\tilde{\forall}^{\st} x)(\exists y\leq^{*}f(x))\varphi(x,y). 
\]

\item Realization $\textsf{R}^{\omega}$:  For any internal $\varphi$, we have   
\[
(\forall x)(\exists^{\st}y)\varphi(x, y)\di (\tilde{\exists}^{\st}z)(\forall x)(\exists y\leq^{*}z)\varphi(x,y). 
\]
\item Majorizability axiom \textsf{MAJ}$_{\st}^{\omega}$:   $(\forall^{\st}x)(\exists^{\st}y)(x\leq^{*} y)$   
\end{enumerate}
\end{defi}
Ferreira and Gaspar define the functional interpretation $U_{\st}$ in \cite{fega}*{\S2} (We introduce it in the proof of Corollary \ref{consresultcor}) and prove the following theorem for the system $\M\equiv \textsf{\textup{E-PA}}^{\omega}_{\st} + \textsf{\textup{mAC}}^{\omega}_{\st} + \textsf{\textup{R}}^{\omega} + \textsf{\textup{MAJ}}^{\omega}_{\st} $.  
\begin{thm}\label{consresult3}
Let $\Psi(\tup a)$ be a formula in the language of \textup{\textsf{E-PA}}$^{\omega}_{\st}$ and suppose $\Psi(\tup a)^{{U_{\st}}}\equiv(\tilde{\forall}^{\st} \tup b)( \tilde{\exists}^{\st} \tup c) \varphi(\tup b, \tup c, \tup a)$. If 
\be\label{antecedn3}
\M\vdash \Psi(\tup a), 
\ee
then there are closed monotone terms $t$ of appropriate types such that
\be\label{consequalty3}
\textup{\textsf{E-PA}}^{\omega}   \vdash ( \tilde{\forall} \tup b)  \varphi(\tup b,\tup t(b), \tup a).
\ee
\end{thm}
\begin{proof}
Immediate by \cite{fega}*{Theorem 1 and Corollary 1}.  
\end{proof}
The notion of `normal form' in this paper will \emph{always} refer to a formula of the form $(\forall^{\st}x)(\exists^{\st}y)\varphi(x, y)$ with $\varphi$ internal, i.e.\ without `st'.  
Note that normal forms are closed under modes ponens inside $\M$.   
We now prove that normal forms are invariant under $U_{\st}$, modulo some strong majorizability predicates.   
\begin{cor}\label{consresultcor}
If for internal $\psi$ the formula $\Psi(\underline{a})\equiv(\forall^{\st}\underline{b})(\exists^{\st}\underline{c})\psi(\underline{x},\underline{y}, \underline{a})$ satisfies \eqref{antecedn3}, then 
$(\tilde{\forall} \underline{b})(\forall x\leq^{*}b)(\exists y\leq^{*}t(b))\psi(\underline{x},\underline{y},\underline{a})$ is proved in \eqref{consequalty3}.  
\end{cor}
\begin{proof}
Clearly, if for $\psi$ and $\Psi$ as given we have 
\be\label{fust}
\Psi(\underline{a})^{U_{\st}}\equiv (\tilde{\forall}^{\st} \tup b)( \tilde{\exists}^{\st} \tup c) (\forall \tup x\leq^{*}b)(\exists \tup y\leq^{*}c )\psi(\tup x, \tup y, \tup a), 
\ee
then the corollary follows immediately from the theorem.  
A tedious but straightforward verification using the clauses (1)-(7) in \cite{fega}*{Def.\ 1} establishes that indeed $\Psi(\underline{a})^{U_{\st}}$ satisfies \eqref{fust}.  
For completeness, we now list these inductive clauses and perform this verification.  
In particular, we now define $U_{\st}$ which assigns to each formula $\Phi$ in the language of $\textsf{E-PA}^{\omega}_{\st}$ two formulas $\Phi^{U_{\st}}$ and $\Phi_{U_{\st}}$ such that $\Phi^{U_{\st}}\equiv (\tilde{\forall}^{\st}  b)( \tilde{\exists}^{\st}  c) \Phi_{U_{\st}}( b,  c)$, with $\Phi_{U_{st}}$ an internal formula, according to the following clauses:
\begin{enumerate}[(i)]
\item $\Phi_{U_{\st}} $ and $\Phi^{U_{\st}}$ are simply $\Phi$, for internal formulas $\Phi$
\item $[\st(x)]^{U_{\st}}\equiv (\tilde{\exists}^{\st}c)(x\leq^{*}c)$.  
\end{enumerate}
For the remaining cases, if the interpretations for $\Phi$ and $\Psi$ are given respectively by $\Phi^{U_{\st}}\equiv (\tilde{\forall}^{\st}  b)( \tilde{\exists}^{\st}  c) \Phi_{U_{\st}}( b,  c)$ and 
$\Psi^{U_{\st}}\equiv (\tilde{\forall}^{\st}  d)( \tilde{\exists}^{\st}  e) \Psi_{U_{\st}}( d,  e)$ then we define:
\begin{enumerate}
\item[(iii)]  $(\Phi\vee \Psi)^{U_{\st}}\equiv (\tilde{\forall}^{\st}  b, d)( \tilde{\exists}^{\st}  c, e) [\Phi_{U_{\st}}( b,  c)\vee \Psi_{U_{\st}}( d,  e)]$,    
\item[(iv)]  $(\neg \Phi)^{U_{\st}}\equiv   (\tilde{\forall}^{\st}f)(\tilde{\exists}^{\st}b)[(\tilde{\exists}b'\leq^{*}b)\neg\Phi_{U_{\st}}(b', f(b'))] $, 
\item[(v)]   $\big((\forall x)\Phi(x)\big)^{U_{\st}}\equiv (\tilde{\forall}^{\st}  b)( \tilde{\exists}^{\st}  c)[(\forall x) \Phi_{U_{\st}}(x,  b,  c)]$, 
\end{enumerate}
where the internal formulas between square brackets are the corresponding lower $U_{\st}$-formulas.  The previous clauses can be used to derive (modulo classical logic) the following interpretations:
\begin{enumerate}
\item[(vi)] $( \Phi\di \Psi)^{U_{\st}}\equiv   (\tilde{\forall}^{\st}f, d)(\tilde{\exists}^{\st}b, e)[(\tilde{\forall}b'\leq^{*}b)\Phi_{U_{\st}}(b', f(b')) \di \Psi_{U_{\st}}(d, e)] $,  
\item[(vii)] $((\exists x)\Phi(x))^{U_{\st}}\equiv (\tilde{\forall}^{\st}F)(\tilde{\exists}^{\st}f)[(\tilde{\exists}f'\leq^{*}f)(\exists x)(\tilde{\forall}b'\leq^{*}F(f') )\Phi_{U_{\st}}(x, b', f'(b')) ]$. 
\end{enumerate}
Additionally, it is noted in \cite{fega}*{\S3} that conjunction can be given the following interpretation:
\begin{enumerate}
\item[(viii)] $(\Phi\wedge \Psi)^{U_{\st}}\equiv (\Phi^{U_{\st}}\wedge\Psi^{U_{\st}})$.  
\end{enumerate}
Thus, let $\psi$ be internal and consider the following formula
\[
(\st(y)\wedge \psi(x,y))^{U_{\st}}\equiv (\tilde{\exists}^{\st}c)[y\leq^{*} c \wedge \psi(x,y)],    
\]
which yields that $\big((\exists y)(\st(y)\wedge \psi(x,y))\big)^{U_{\st}}$ is exactly
\be\label{dongel}
 (\tilde{\forall}^{\st}F)(\tilde{\exists}^{\st}f)\big[(\tilde{\exists}f'\leq^{*}f)(\exists y)(\tilde{\forall}b'\leq^{*}F(f') )[y\leq^{*} f'(b')\wedge \psi(x,y)] \big].
\ee
Take $F_{0}$ and the associated $f_{0}$ as in \eqref{dongel}, and define $e_{0}:=f_{0}(F_{0}(f_{0}))$.  Using the monotonicity of $F_{0}$ and $ f_{0}$ and the transitivity of $\leq^{*}$, we obtain: 
\be\label{dingel}
\big((\exists y)(\st(y)\wedge \psi(x,y))\big)^{U_{\st}}\di (\tilde{\exists}^{\st}e_{0})(\exists y\leq^{*}e_{0})\psi(x,y).    
\ee
For the reverse implication in \eqref{dingel}, for any monotone and standard $F$, take $f, f'$ to be the function which is constant $e_{0}$, the latter originating from the consequent of \eqref{dingel}.   
In this way, we obtain \eqref{dongel} and an equivalence in \eqref{dingel}.  Next, note that $[\st(x)]^{U_{\st}}\equiv (\tilde{\exists}^{\st}c)(x\leq^{*}c)$, and $\big[\st(x)\di (\exists y)(\st(y)\wedge \psi(x,y))\big]^{U_{\st}}$ is thus
\[
 (\tilde{\forall}^{\st}f, d)(\tilde{\exists}^{\st}b, e_{0})[(\tilde{\forall}b'\leq^{*}b)(x\leq^{*}f(b')\di (\exists y\leq^{*}e_{0})\psi(x,y)) ].     
\]  
Finally, $\big[(\forall x)(\st(x)\di (\exists y)(\st(y)\wedge \psi(x,y))\big]^{U_{\st}}$ is the formula
\be\label{diverg}
 (\tilde{\forall}^{\st}f )(\tilde{\exists}^{\st}b, e_{0})(\forall x)[(\tilde{\forall}b'\leq^{*}b)(x\leq^{*}f(b')\di (\exists y\leq^{*}e_{0})\psi(x,y)) ].     
\ee
Now take some standard and monotone $x_{0}$ of the same type as the variable $x$ and apply \eqref{diverg} to $f_{0}$, the constant $x_{0}$ function.   
We obtain 
\[
 (\tilde{\forall}^{\st}x_{0} )(\tilde{\exists}^{\st} e_{0})(\forall x)[(x\leq^{*}x_{0}\di (\exists y\leq^{*}e_{0})\psi(x,y)) ],      
\]
which is exactly as required, and we are done.  
\end{proof}
In light of the previous theorem, normal forms are invariant under $U_{\st}$ modulo some 
extra appearances of the {strong majorizability} predicate.  As noted in \cite{fega}*{\S2}, objects of type zero and one are strongly majorisable, and are thus not affected (much) by the aforementioned addition of the strong majorizability predicate.  
Thus, it seems the results from \cite{sambon} could also go through for $\M$ as long the standard variables in the normal form are of low type.  We shall make this more precise in the following sections.      

\medskip

Finally, as studied in \cite{fega}*{\S4}, the bounded functional interpretation by Ferreira and Oliva from \cite{folieke} can be seen as a nonstandard interpretation where the nonstandard part is concealed.  

\subsection{Notations in $\M$}
In this section, we introduce notations relating to $\M$.  

\medskip

First of all, we use the same notations as in \cite{fega}, some of which we repeat here.  
\begin{rem}[Notations]\label{notawin}\rm
We write $(\forall^{\st}x^{\tau})\Phi(x^{\tau})$ and $(\exists^{\st}x^{\sigma})\Psi(x^{\sigma})$ as short for 
$(\forall x^{\tau})\big[\st(x^{\tau})\di \Phi(x^{\tau})\big]$ and $(\exists^{\st}x^{\sigma})\big[\st(x^{\sigma})\wedge \Psi(x^{\sigma})\big]$.     
We also write $(\forall x^{0}\in \Omega)\Phi(x^{0})$ and $(\exists x^{0}\in \Omega)\Psi(x^{0})$ as short for 
$(\forall x^{0})\big[\neg\st(x^{0})\di \Phi(x^{0})\big]$ and $(\exists x^{0})\big[\neg\st(x^{0})\wedge \Psi(x^{0})\big]$.  Furthermore, if $\neg\st(x^{0})$ (resp.\ $\st(x^{0})$), we also say that $x^{0}$ is `infinite' (resp.\ finite) and write `$x^{0}\in \Omega$'.  
Finally, a formula $A$ is `internal' if it does not involve $\st$, and $A^{\st}$ is defined from $A$ by appending `st' to all quantifiers (except bounded number quantifiers).    
\end{rem}
Secondly, we will use the usual notations for rational and real numbers and functions as introduced in \cite{kohlenbach2}*{p.\ 288-289} (and \cite{simpson2}*{I.8.1} for the former).  
We shall (sparingly) use the \emph{symbolic} notation `$\R$' for the set of real numbers.  
\begin{defi}[Real number]\label{keepintireal}\rm
A (standard) real number $x$ is a (standard) fast-converging Cauchy sequence $q_{(\cdot)}^{1}$, i.e.\ $(\forall n^{0}, i^{0})(|q_{n}-q_{n+i})|<_{0} \frac{1}{2^{n}})$.  
We freely make use of Kohlenbach's `hat function' from \cite{kohlenbach2}*{p.\ 289} to guarantee that every sequence $f^{1}$ can be viewed as a real.  We also use the notation $[x](k):=q_{k}$ for the $k$-th approximation of real numbers.    
Two reals $x, y$ represented by $q_{(\cdot)}$ and $r_{(\cdot)}$ are \emph{equal}, denoted $x=_{\R}y$, if $(\forall n)(|q_{n}-r_{n}|\leq \frac{1}{2^{n}})$. Inequality $<_{\R}$ is defined similarly and the subscript $\R$ is often omitted.          
Functions from reals to reals are represented by functionals $\Phi^{1\di 1}$ such that 
\be\label{REXT}\tag{\textsf{RE}}
(\forall x^{1}, y^{1})(x=_{\R}y\di \Phi(x)=_{\R}\Phi(y)), 
\ee
i.e.\ equal reals are mapped to equal reals.  The property \eqref{REXT} is abbreviated by $\Phi:\R\di \R$.  We also write $x\approx y$ if $(\forall^{\st} n)(|q_{n}-r_{n}|\leq \frac{1}{2^{n}})$ and $x\gg y$ if $x>y\wedge x\not\approx y$.  
Finally, sets are denoted $X^{1}, Y^{1}, Z^{1}, \dots$ and are given by their characteristic functions $f^{1}_{X}$, i.e.\ $(\forall x^{0})[x\in X\asa f_{X}(x)=1]$, where $f_{X}^{1}$ is assumed to be binary.      
\end{defi}
Thirdly, we use the usual extensional notion of equality.  
\begin{rem}[Equality]\label{equ}\rm
The system $\M$ includes equality between natural numbers `$=_{0}$' as a primitive.  Equality `$=_{\tau}$' for type $\tau$-objects $x,y$ is then defined as follows:
\be\label{aparth}
[x=_{\tau}y] \equiv (\forall z_{1}^{\tau_{1}}\dots z_{k}^{\tau_{k}})[xz_{1}\dots z_{k}=_{0}yz_{1}\dots z_{k}]
\ee
if the type $\tau$ is composed as $\tau\equiv(\tau_{1}\di \dots\di \tau_{k}\di 0)$.
In the spirit of Nonstandard Analysis, we define `approximate equality $\approx_{\tau}$' as follows:
\be\label{aparth2}
[x\approx_{\tau}y] \equiv (\forall^{\st} z_{1}^{\tau_{1}}\dots z_{k}^{\tau_{k}})[xz_{1}\dots z_{k}=_{0}yz_{1}\dots z_{k}]
\ee
with the type $\tau$ as above.  
All systems under consideration include the \emph{axiom of extensionality} for all $\varphi^{\rho\di \tau}$ as follows:
\be\label{EXT}\tag{\textsf{E}}  
(\forall  x^{\rho},y^{\rho}) \big[x=_{\rho} y \di \varphi(x)=_{\tau}\varphi(y)   \big].
\ee
However, as noted in \cite{brie}*{p.\ 1973}, the so-called axiom of \emph{standard} extensionality \eqref{EXT}$^{\st}$ is problematic.  The latter cannot be included in $\M$ for the same reason this axiom cannot be included in the systems from \cite{brie}.  
Finally, we always refer to \eqref{REXT} from the previous remark as `extensionality on the reals' to distinguish it from \eqref{EXT}.  
  
\end{rem}

\subsection{Non-standard Nonstandard Analysis in $\M$}\label{NSANSA}
In this section, we prove some (mostly non-classical) basic theorems inside the system $\M$.  
\begin{thm}[$\M$]\label{OMG}
Every binary sequence is standard.  
The same holds for sequences of binary sequences. 
Every sequence $\alpha^{1}$ which is standard for standard index, i.e.\ $(\forall^{\st}n)(\exists^{\st}m)(\alpha(n)=m)$, has a standard part, i.e.\ $(\exists^{\st}\beta^{1})(\alpha\approx_{1} \beta)$.  
\end{thm}
\begin{proof}
Every binary sequence $\alpha^{1}$ by definition satisfies $\alpha\leq_{1}11\dots$, and the latter implies $\alpha\leq_{1}^{*}11\dots$ by definition. 
As $11\dots$ is standard (which follows from the fact that all recursor constants are standard), the second basic axiom of $\M$ implies that $\alpha$ is also standard.  
Similarly, every sequence of binary sequences is majorized by the sequence of sequences which outputs $11\dots$ for every input.  Finally, let $\alpha^{1}$ be a sequence which is standard for standard indices.  
Define the (standard by the previous) binary sequence of sequences $\beta^{0\di 1}$ by $\beta(n, m)=1$ if $\alpha(n)=m$, and zero otherwise.  
Now apply $ \textsf{mAC}_{\st}^{\omega}$ to $(\forall^{\st}n)(\exists^{\st}m)\beta(n, m)=1$ to obtain standard $f^{1}$ such that $(\forall^{\st}n)(\exists m\leq f(n))\beta(n, m)=1$.  It is now trivial to define the standard part of $\alpha$ using $f$ and $\beta$.  
\end{proof}
Let $\WKL$ be the well-known \emph{weak K\"onig's lemma} (See e.g.\ \cite{simpson2}*{IV}).  
\begin{cor}\label{fukwkl}
The system $\M$ proves $\WKL^{\st}$.  
\end{cor}
\begin{proof}
Let $T$ be any binary tree such that $(\forall^{\st}n)(\exists \alpha)(|\alpha|=n\wedge \alpha\in T)$.  
By overflow (See Theorem \ref{koverflow} or use minimisation), there is a binary sequence $\sigma$ of nonstandard length in $T$, and $\sigma*00\dots$ is standard by the theorem, i.e.\ the consequent of $\WKL^{\st}$ follows for $T$.  
\end{proof}

\begin{cor}[$\M$]\label{XYX}
Every standard extensional type two functional is \emph{nonstandard} uniformly continuous on Cantor space, i.e.\ $(\forall \varphi^{2})(\forall f, g\leq_{1}1)(f\approx_{1}g\di \varphi(f)=\varphi(g))$.
The same holds for any standard compact subspace of Baire space.  
\end{cor}
\begin{proof}
The first part is immediate as the theorem implies that every binary sequence is standard.  As a standard compact subspace of Baire space is given by $\{f^{1}:f\leq_{1}h_{0} \}$ for some fixed standard $h_{0}$, the second part now follows in the same way as in the proof of the theorem.  
\end{proof}
By way of example, every function on Cantor space which is continuous in the sense of Reverse Mathematics (See \cite{simpson2}*{II.6.2}) relative to `st', is also standard extensional.   

\medskip

By the previous results, it is clear\footnote{Apply $\paai$ to $\alpha_{0}\approx_{1}00\dots $ for $\alpha_{0}(n)=0$ for $n\leq M\in \Omega$ and $1$ otherwise.} that $\M+\paai$ is inconsistent.  The latter is a special case of Nelson's axiom \emph{Transfer}, and defined as follows:
\be\tag{$\paai$}
(\forall^{\st}f^{1})\big[(\forall^{\st}n^{0})f(n)=0\di (\forall m)f(m)=0\big].
\ee
By \cite{bennosam}*{Cor.\ 12}, $\paai$ is equivalent to the following `Turing jump functional':
\be\tag{$\exists^{2}$}  
(\exists \varphi^{2})(\forall f)[\varphi(f)=0\asa (\exists n)f(n)=0].
\ee
Unsurprisingly, $(\exists^{2})$ is inconsistent with Kohlenbach's uniform boundedness principles from \cite{kohlenbach3}*{\S12}, and $\M$ validates the latter principles as noted in \cite{fega}*{\S1-2}.  

\medskip

From the Reverse Mathematics of weak K\"onig's lemma (See \cite{simpson2}*{IV}), we know that uniform continuity and the existence of an upper bound are equivalent (for continuous functions on a compact space).  
Thus, the following theorem should not come as such a surprise in light of Corollary \ref{XYX}, although there is no continuity requirement. 
\begin{thm}[$\M$]\label{ciro}
Every standard $\varphi^{2}$ has a standard upper bound on Cantor space \(or on any standard compact space\).   
\end{thm}
\begin{proof}
The axiom $\textsf{MAJ}^{\omega}_{\st}$ guarantees that every standard $\varphi^{2}$ has a standard majorant, i.e.\ $(\forall^{\st}\varphi^{2})(\exists^{\st}\psi^{2})(\varphi\leq_{2}^{*}\psi)$.  By definition, the latter formula implies that 
\[
(\forall^{\st}\varphi^{2})(\exists^{\st}\psi^{2})(\forall f^{1}\leq^{*}_{1}g )(\varphi(f)\leq_{0}\psi(g)),   
\]
and taking $n_{0}=\psi(g_{0})$ for $g_{0}=11\dots$, we obtain $(\forall^{\st}\varphi^{2})(\exists^{\st}n_{0})(\forall f^{1}\leq_{1}1 )(\varphi(f)\leq_{0}n_{0})$,   
which is exactly as required.  
\end{proof}
Next, we prove a version of the Standard part principle for type two functionals.  
Define $\b^{1\di 1}$ by $\b(f)(k)=0$ if $f(k)=0$ and $1$ otherwise;  We refer to $\varphi(\b(\cdot))$ as \emph{the restriction to Cantor space} of $\varphi^{2}$.  
\begin{thm}[$\M$]\label{soohee}
Every type two functional restricted to Cantor space which maps standard binary sequences to standard numbers, is standard.  
\end{thm}
\begin{proof}
For $\varphi^{2}$ as in the theorem, we have $(\forall f^{1})(\exists^{\st}n)(\varphi(\b(f))\leq n)$, as every binary sequence is standard in $\M$.  
Apply $\textsf{R}^{\omega}$ to obtain a standard upper bound on $\varphi(\b(\cdot))$, say $n_{0}$.  Then it is easy to verify that $\varphi(\b(\cdot)) \leq_{2}^{*}n_{0}^{2}$, and since the latter constant function is standard, so is $\varphi(\b(\cdot))$ 
by the second basic axiom of $\M$.    
\end{proof}
Note that the previous result is classically false and much stronger than the Standard Part principle of $\IST$.  Combining with Theorem \ref{ciro}, every type two functional which does not attain infinite values on the standard part of Cantor space, has a standard upper bound.  

\medskip

Finally, the system $\M$ proves the very useful \emph{overflow} principle, which is not non-standard in nature.  
\begin{thm}[$\M$]\label{koverflow}
If $(\forall^{\st}x^{\sigma})\varphi(x)$ for internal $\varphi$, then there is a \emph{nonstandard} $y^{\sigma}$ such that $\varphi(y)$. 
\end{thm}
\begin{proof}
As noted in \cite{fega}*{\S2}, the axiom $\textsf{R}^{\omega}$ implies the idealisation axiom.  The latter is used in \cite{brie}*{Prop.\ 3.3} to prove overflow.  
\end{proof}

\section{Proof mining using $\M$}\label{kafir}  
In this section, we establish two important features of the system $\M$.  First of all, we show that $\M$ allows us to extract computational information from proofs in Nonstandard Analysis, similar (but with significant modifications) to the results in \cite{sambon}.  
Secondly, we show how the results from Section \ref{NSANSA} allow us to extract computational information from proofs \emph{not involving Nonstandard Analysis}, as explained in Section \ref{wintro}.       
\subsection{Mining a nonstandard proof}\label{infidel}
In this section, we study the statement $\CRI$: \emph{A uniformly continuous function on the unit interval is Riemann integrable}.  
We shall extract a proof of the effective version \eqref{EST} from a proof of the \emph{nonstandard} version \eqref{NST} of $\CRI$ inside $\M$.  
In this way, we show that the techniques from \cite{sambon} can work for the system $\M$, with some non-trivial modification.  

\medskip

The previous observation is highly relevant for us, as the {extraction} of computational information from proofs and theorems \emph{not involving Nonstandard Analysis} will still proceed by a detour via Nonstandard Analysis,
as discussed in Section \ref{wintro}, and implemented in Section~\ref{examinisalamini}. 

\medskip

The relevant definitions of continuity and integration are as follows.
\bdefi[Continuity]
A function $f:\R\di \R$ is \emph{nonstandard continuous} on $[0,1]$ if
\be\label{soareyou3}
(\forall^{\st}x\in [0,1])(\forall y\in [0,1])[x\approx y \di f(x)\approx f(y)].
\ee
A function $f:\R\di \R$ is \emph{nonstandard uniformly continuous} on $[0,1]$ if
\be\label{soareyou4}
(\forall x, y\in [0,1])[x\approx y \di f(x)\approx f(y)].
\ee
A function $f:\R\di \R$ is \emph{uniformly continuous} on $[0,1]$ if
\be\label{soareyou6}\textstyle
(\forall k^{0})(\exists N^{0})(\forall x^{1}, y^{1}\in [0,1])[ |x -y|<\frac{1}{N} \di |f(x)- f(y)|\leq\frac1k].
\ee
A \emph{modulus} of uniform continuity outputs $N^{0}$ as in \eqref{soareyou6} on input any $k^{0}$
\edefi
\bdefi[Integration]
For a partition $\pi=(0, t_{0}, x_{1},t_{1},  \dots,x_{M-1}, t_{M-1}, 1)$ of the unit interval (which we denote by $\pi \in P([0,1]) $), the number $S_{\pi}(f):=\sum_{i=0}^{M-1}f(t_{i}) (x_{i}-x_{i+1}) $ is the \emph{Riemann sum} and  
and $\|\pi\|$ is the \emph{mesh}, i.e.\ the largest distance between two adjacent partition points $x_{i}$ and $x_{i+1}$. 
A function $f:\R\di \R$ is \emph{nonstandard \(Riemann\) integrable} on $[0,1]$ if
\be\label{soareyou5}
(\forall \pi, \pi' \in P([0,1]))\big[\|\pi\|,\| \pi'\|\approx 0  \di S_{\pi}(f)\approx S_{\pi}(f)  \big],
\ee
A function $f:\R\di \R$ is \emph{\(Riemann\) integrable} on $[0,1]$ if
\be\label{soareyou7}\textstyle
(\forall k^{0})(\exists N)(\forall \pi, \pi' \in P([0,1]))\big[\|\pi\|,\| \pi'\| <\frac{1}{N}  \di |S_{\pi}(f)- S_{\pi}(f)|<\frac1k  \big],
\ee

\edefi
Let $\textsf{CRI}_{\ns}$ be the statement $(\forall f:\R\di \R)[\eqref{soareyou4}\di \eqref{soareyou5}]$, define $\textsf{CRI}_{\ef}(t)$ as \eqref{EST}.  
\begin{thm}\label{varou}
From the proof of $\CRI_{\ns}$ in $\M$ a term $t$ can be extracted such that $\textup{\textsf{E-PA}}^{\omega} $ proves $\CRI_{{\ef}}(t)$.  
\end{thm}
\begin{proof}
For completeness, we first show that $\CRI_{\ns}$ can be proved in $\M $.  A similar proof may be found in \cites{sambon, aloneatlast3}.  

\medskip
  
Given two partitions of the unit interval $\pi=(0, t_{0}, x_{1},t_{1},  \dots,x_{M-1}, t_{M-1}, 1)$ and 
$\pi'=(0, t_{0}', x_{1}',t_{1}',  \dots,x_{M'-1}', t_{M'-1}, 1)$ with infinitesimal mesh, let $x_{i}''$ for $i\leq M''$ be an enumeration\footnote{To make this enumeration effective, work with the approximations $[x_{i}](2^{M})$ and $[x_{i}'](2^{M'})$ and note that the difference is infinitesimal in the below steps.} of all $x_{i}$ and $x_{i}'$ in increasing order.  Let $t_{i}''$ and $t_{i}'''$ for $i\leq M''$ be the associated $t_{i}$ and $t_{i}''$ with repetitions (corresponding to $x_{i}''$) of the latter to obtain a list of length $M''$.    
Then we have that
\begin{align*}\textstyle
|S_{\pi}(f)-S_{\pi'}(f)| &\textstyle=|\sum_{i=0}^{M-1}f(t_{i}) (x_{i}-x_{i+1}) -\sum_{i=0}^{M'-1}f(t_{i}') (x_{i}'-x_{i+1}') |\\
&\textstyle=|\sum_{i=0}^{M''-1}f(t_{i}'') (x_{i}''-x_{i+1}'') -\sum_{i=0}^{M''-1}f(t_{i}''') (x_{i}''-x_{i+1}'') |\\
&\textstyle=|\sum_{i=0}^{M''-1}(f(t_{i}'')-f(t_{i}'''))\cdot (x_{i}''-x_{i+1}'') |\\
&\textstyle\leq \sum_{i=0}^{M''-1}|f(t_{i}'')-f(t_{i}''')| \cdot|x_{i}''-x_{i+1}''| \leq   \sum_{i=0}^{M''-1}\eps_{0} \cdot|x_{i}''-x_{i+1}''|\approx 0,
\end{align*}
where $\eps_{0}:=\max_{i\leq M''}|f(t_{i}'')-f(t_{i}''')|$ is an infinitesimal due to the (uniform nonstandard) continuity of $f$ and the definition of $\pi, \pi'$.  

\medskip

Hence, we obtain $\CRI_{\ns}$ inside $\M$, from which we now derive $\CRI_{\ef}(t)$.  To this end, we show that $\CRI_{\ns}$ can be brought into the normal form for applying Corollary~\ref{consresultcor}.  
First of all, we resolve `$\approx$' in the antecedent of $\CRI_{\ns}$ as follows:
\be\label{first}\textstyle
(\forall x, y\in [0,1])[(\forall^{\st}N)|x- y|\leq \frac{1}{N} \di (\forall ^{\st}k)|f(x)- f(y)|\leq \frac{1}{k}].
\ee
Bringing all standard quantifiers in \eqref{first} outside the square brackets, we obtain:
\be\label{first2}\textstyle
(\forall^{\st}k)(\forall x, y\in [0,1])(\exists^{\st}N)[|x- y|\leq \frac{1}{N} \di |f(x)- f(y)|\leq \frac{1}{k}].
\ee
Since the formula in square brackets in \eqref{first2} is internal, we apply $\textsf{R}^{\omega}$ and obtain 
\be\label{first3}\textstyle
(\forall^{\st}k)(\exists^{\st}N')(\forall x, y\in [0,1])(\exists N\leq^{*} N')[|x- y|\leq \frac{1}{N} \di |f(x)- f(y)|\leq \frac{1}{k}].
\ee
By definition, $\leq^{*}_{0}$ is just $\leq_{0}$, implying
\be\label{first4}\textstyle
(\forall^{\st}k)(\exists^{\st}N)\big[(\forall x, y\in [0,1])[|x- y|\leq \frac{1}{N} \di |f(x)- f(y)|\leq \frac{1}{k}]\big].
\ee
Since the formula in (big) square brackets is internal, we may apply $\MAC$ and obtain a standard and monotone $g^{1}$ such that:
\be\label{first667}\textstyle
(\forall^{\st}k)(\exists N\leq^{*} g(k))\big[(\forall x, y\in [0,1])[|x- y|\leq \frac{1}{N} \di |f(x)- f(y)|\leq \frac{1}{k}]\big].
\ee
Again using that $\leq_{0}^{*}$ is $\leq_{0}$, we obtain
\be\label{first668}\textstyle
(\forall^{\st}k)\big[(\forall x, y\in [0,1])[|x- y|\leq \frac{1}{g(k)} \di |f(x)- f(y)|\leq \frac{1}{k}]\big],
\ee
i.e.\ $g$ is a modulus of continuity for $f$.      
Similarly, the consequent of $\CRI_{\ns}$ yields:
\be\textstyle\label{second}
(\forall^{\st}k')(\exists^{\st}N')\big[(\forall \pi, \pi' \in P([0,1]))(\|\pi\|,\| \pi'\|\leq \frac{1}{N'}  \di |S_{\pi}(f)- S_{\pi}(f)|\leq \frac{1}{k'} )\big]
\ee
where $B(k', N', f)$ is the (internal) formula in square brackets;  Let $A(g, k, f)$ be the (internal) formula in big square brackets in \eqref{first668}.   Then $\CRI_{\ns}$ implies 
\be\label{second2}
(\tilde{\forall}^{\st}  g, k')(\forall f)(\exists^{\st} N', k)[A(g, k, f)\di B(k', N', f)], 
\ee
in which the formula in square brackets is internal again.  Note that we may omit the limitation to monotone $g$, since \eqref{first668} also implies \eqref{first} without this monotonicity assumption.  
Thus, we obtain
\be\label{second266}
({\forall}^{\st}  g, k')(\forall f)(\exists^{\st} N', k)[A(g, k, f)\di B(k', N', f)], 
\ee
Applying \textsf{R}$^{\omega}$ to \eqref{second266}, we obtain 
\be\label{second22}
(\forall^{\st}  g, k')(\exists^{\st}l)(\forall f)(\exists N', k\leq^{*} l)[A(g, k, f)\di B(k', N', f)], 
\ee  
which yields, by the special structure of $B$ and the definition of $\leq_{0}^{*}$, that:
\be\label{second23}
(\forall^{\st}  g, k')(\exists^{\st}N')(\forall f)(\exists k)[A(g, k, f)\di B(k', N', f)].
\ee  
Applying Corollary \ref{consresultcor} to \eqref{second23}, the system $\textup{\textsf{E-PA}}^{\omega}$ proves
\[
(\tilde{\forall}   g, {k}')(\forall g_{0}, k_{0}' \leq^{*}g, k')(\exists N'\leq^{*}  t( g_{0}, k_{0}'))(\forall f)(\exists k)[A(g_{0}, k, f)\di B(k'_{0}, N', f)],
\]
for some term $t$ from the original language.  For a sequence $g^{1}$, define the corresponding \emph{monotone} $\tilde{g}$ by $\tilde{g}(k):=\max_{n\leq k}g(n)$ and note that $g\leq^{*}\tilde{g}$.  
Since $\leq^{*}_{0}$ is just $\leq_{0}$, the previous yields:
\be\label{crux}
(\forall  f, g)[(\forall k)A(g, k, f)\di (\forall k')B(k', t( \tilde{g}, k'), f)], 
\ee
which is exactly $\CRI_{\ef}(u)$ for $u(g, k')=t(\tilde{g}, k')$ by the definitions of $A$ and $B$.  
\end{proof}
Note that the actual computation in $\CRI_{\ef}(t)$ only takes place on the modulus. 
Thanks to the previous theorem, we can define (inside \textsf{E-PA}$^{\omega}$) \emph{the} Riemann integral functional $I(f,0, x)$ 
which takes as input a continuous function $f$ on $[0,1]$ and outputs $\int_{0}^{x}f(t) dt$.  In particular, $[I(f,0, x)](k)$ is defined as $\sum_{i=0}^{i(x)} f(\frac{i}{2^{k}}) \frac{1}{2^{k}} $ where $i(x)$ 
is the least partition point larger than $x$.  To guarantee that $I(f,0, x)$ converges `fast enough', one considers $[I(f,0, x)](t(\tilde{g},k))$ where $g$ is a modulus of uniform continuity of $f$ and $t$ is the term from Theorem~\ref{varou}.  
%
%
%

\subsection{Mining a standard proof}\label{examinisalamini}
In this section, we treat the \emph{fundamental theorem of calculus} (denoted $\FTC$; See \cite{rudin}*{\S6}), which states that differentiability and integration cancel each other out.  

\medskip

We shall obtain a proof of the effective version \eqref{kinkel} of $\FTC$ from a proof of the latter in `standard' mathematics, i.e.\ not involving Nonstandard Analysis.  
We shall accomplish this by taking the following `detour' into Nonstandard Analysis:  We prove in $\M$ that the `usual' $\eps$-$\delta$-definition (of continuity, differentiability, et cetera) relative to `st', is equivalent to the associated \emph{nonstandard} definition.  
In this way, we observe that $\FTC^{\st}$ implies a \emph{nonstandard} version of $\FTC$, namely \eqref{dorkkkk} below.  To the latter, we apply the `term extraction' techniques from the previous section to obtain the effective version \eqref{kinkel} of $\FTC$.     

\medskip
    
The previous paragraph highlights an advantage of $\M$ over the systems from \cite{brie}, namely that we can extract information from proofs \emph{not involving Nonstandard Analysis}.  
By contrast, the systems from \cite{brie} were used in \cite{sambon} to extract effective versions \emph{from proofs in Nonstandard Analysis}.
More examples are studied in Section~\ref{BIS}, while a general template is introduced in Section~\ref{sucktwo}.  

\medskip

We first study the first part of $\FTC$ which states that $F'(x)=f(x)$ for continuous $f$ such that $F(x):=\int_{a}^{x}f(x)dx$;   
We shall consider the following `usual' version of $\FTC$, which states that for all $l^{0}$ and $f:\R\di \R$, we have
\begin{align}  
(\forall k')&(\exists N')(\forall \textstyle x, y \in [0,1])(|x-y|<\frac{1}{N'}\di |f(x)-f(y)|\leq\frac{1}{k'}]\di \notag\\
&\textstyle(\forall k)(\exists N)(\forall \eps)(\forall  x\in [\frac{1}{l}, 1-\frac{1}{l}])\big[\eps<\frac{1}{N}  \di |\Delta_{\eps}(I(f,0, x))- f(x)|\leq \frac{1}{k}],\label{furg} 
\end{align}
%
where $\Delta_{\eps}(f(x)):=\frac{f(x+\eps)-f(x)}{\eps}$.  
The effective version $\FTC_{\ef}(t)$ is the statement
\begin{align}  
(\forall f, g, l)\textstyle\Big[&\textstyle(\forall \textstyle x, y \in [0,1], k')(|x-y|<\frac{1}{g(k')}\di |f(x)-f(y)|\leq\frac{1}{k'}]\di\label{kinkel} \\
&\textstyle(\forall k)\big(\forall N\geq t(g , k, l), x\in \big[\frac{1}{l}, 1-\frac{1}{l}\big]\big)\big[|\Delta_{\frac{1}{N}}(I(f,0, x))- f(x)|\leq \frac{1}{k}\big]\Big].\notag
\end{align}
Note that Kohlenbach has shown that continuous real-valued functions as represented in Reverse Mathematics (See \cite{simpson2}*{II.6.6} and \cite{kohlenbach3}*{Prop.\ 4.4}) automatically have a modulus of pointwise continuity.  
The same construction works for uniformly continuous functions and the associated modulus of uniform continuity.  Hence, the antecedent of $\FTC_{\ef}(t)$ is still quite natural.        
\begin{thm}\label{floggen4}
From the proof of \eqref{furg} in $\textup{\textsf{E-PA}}^{\omega} $, a term $t$ can be extracted such that $\textup{\textsf{E-PA}}^{\omega} $ proves $\FTC_{{\ef}}(t)$.  
\end{thm}
\begin{proof}
It is straightforward to prove \eqref{furg} inside \textsf{E-PA}$^{\omega}$;  It is then easy\footnote{The original proof actually goes through verbatim relative to `st', as standard extensionality follows from $\eps$-$\delta$-continuity relative to `st'.} to verify that $\M$ proves \eqref{furg}$^{\st}$ for any $f:\R\di \R$ and standard $l^{0}$.  We now derive the nonstandard definitions of (uniform) continuity and differentiability from the usual $\eps$-$\delta$-definitions relative to `st' in $\M$.   

\medskip

First of all, working in $\M$, the definition of uniform $\eps$-$\delta$-continuity relative to `\st' implies the following:
\[
(\forall^{\st} k')(\exists^{\st} N')(\forall^{\st} \textstyle \alpha^{1}, \beta^{1} \leq_{1}1)(|\r(\alpha)-\r(\beta)|\ll \frac{1}{N'}\di |f(\r(\alpha))-f(\r(\beta))|\ll\frac{1}{k'}], 
\]
where $\r(\gamma):=\sum_{n=1}^{\infty}\frac{\gamma(n)}{2^{n}}$.    
Replace the predicates `$\ll$' by `$<$' to obtain
\[
(\forall^{\st} k')(\exists^{\st} N')(\forall^{\st} \textstyle \alpha^{1}, \beta^{1} \leq_{1}1)(|\r(\alpha)-\r(\beta)|< \frac{1}{N'}\di |f(\r(\alpha))-f(\r(\beta))|<\frac{1}{k'}]. 
\]
By Theorem \ref{OMG}, every binary sequence is standard in $\M$, hence
\be\label{vollenback}
(\forall^{\st} k')(\exists^{\st} N')(\forall \textstyle \alpha^{1}, \beta^{1} \leq_{1}1)(|\r(\alpha)-\r(\beta)|< \frac{1}{N'}\di |f(\r(\alpha))-f(\r(\beta))|<\frac{1}{k'}]. 
\ee
Recall that the classical fact \emph{every real number has a binary expansion}, can be proved in $\RCA_{0}$, a weak system of second-order arithmetic (\cite{polahirst}). 
Furthermore, the definition of $f:\R\di \R$ as in Definition \ref{keepintireal}, is such that for a real $x$ and its binary expansion $\alpha$, we have $f(x)=_{\R}f(\r(\alpha))$ as $x=_{\R}\r(\alpha)$.  
Hence, \eqref{vollenback} becomes: 
\be\label{normaf}
(\forall^{\st} k')(\exists^{\st} N')\big[(\forall \textstyle x,y\in [0,1])(|x-y|< \frac{1}{N'}\di |f(x)-f(y)|<\frac{1}{k'})\big], 
\ee
which also implies that $f$ is nonstandard uniformly continuous.  Thus, we proved that continuity as in \eqref{soareyou6}$^{\st}$ implies nonstandard continuity for any $f:\R\di \R$.  

\medskip

Secondly, we can similarly prove that for standard $l$, the consequent of \eqref{furg} relative to `st', i.e.\
\be\label{pothead}\textstyle
(\forall^{\st} k)(\exists^{\st} N)(\forall^{\st} \eps, x\in [\frac{1}{l}, 1-\frac{1}{l}])\big[0\ne\eps\ll \frac{1}{N}  \di |\Delta_{\eps}(I(f,0, x))- f(x)|\ll \frac{1}{k}],
\ee
implies the following two nonstandard versions:
\be\label{normag}\textstyle
(\forall^{\st} k)(\exists^{\st} N)\big[(\forall \eps, x\in [\frac{1}{l}, 1-\frac{1}{l}])\big(0\ne\eps< \frac{1}{N}  \di |\Delta_{\eps}(I(f,0, x))- f(x)|< \frac{1}{k}\big)\big],
\ee 
\be\label{knsa}\textstyle
(\forall \eps)(\forall  x\in [\frac{1}{l}, 1-\frac{1}{l}])\big[0\ne\eps\approx 0  \di |\Delta_{\eps}(I(f,0, x))- f(x)|\approx 0],
\ee
where the latter is just the consequent of \eqref{furg} formulated with the nonstandard definition of differentiability.  The formula \eqref{knsa} is easily seen to yield \eqref{normag} in the same way as 
in the proof of Theorem \ref{varou}, namely resolve `$\approx$' and use $\textsf{R}^{\omega}$ to bring all standard quantifiers outside.    

\medskip

Finally, putting the previous together, the proof of \eqref{furg}$^{\st}$ (for any $f:\R\di \R$ and standard $l^{0}$) yields
\be\label{dorkkkk}
(\forall f:\R\di \R)(\forall^{\st}l)\big[(\forall^{\st} k')(\exists^{\st} N')A(k', N', f)\di (\forall^{\st} k)(\exists^{\st} N)B(k, N,f )\big], 
\ee
where $A$ (resp.\ $B$) is the square-bracketed in \eqref{normaf} (resp.\ \eqref{normag}).  
In the same way as in the proof of Theorem \ref{varou}, one now brings \eqref{dorkkkk} into normal form, and applying Corollary \ref{consresultcor} yields $\FTC_{\ef}(t)$, and we are done.  
\end{proof}
Note that it is essential for the proof of the theorem that $f$ is \emph{extensional for real numbers} as in \eqref{REXT}.  Indeed, the latter allows one to 
to pass from \eqref{vollenback} to \eqref{normaf} via binary representations.  
Furthermore, while every binary sequence is standard in $\M$, the latter only implies that every real number \emph{equals} (in the sense of $=_{\R}$) a standard number, but \emph{not} that every real number \emph{is} standard.  
Finally, it would be interesting to compare the complexity of the term from the previous theorem and the term extracted from the `nonstandard' proof in \cite{sambon}*{\S3.3.3}.  

\medskip

In the following remark, we point out that the \emph{second part} of $\FTC$ as formulated in \cite{rudin}*{Theorem~6.21} can be treated similarly.  
\begin{rem}[Second part of $\FTC$]\rm
The nonstandard, non-effective, and effective versions of the second part of $\FTC$ are as follows, using the same notion of differentiability as above.
\begin{align}  
(\forall f)(\forall^{\st} k)\textstyle\Big[(\forall \eps, \eps'\approx 0)&\textstyle(\forall x \in [-\frac{1}{k},1+\frac{1}{k}])[\Delta_{\eps}(f(x))\approx \Delta_{\eps'}(f(x))]\di \notag\\
&\textstyle(\forall \eps> 0)\big[\eps\approx 0 \di I(\Delta_{\eps}(f(\cdot)), 1))\approx f(1)-f(0)\big]\Big].\label{holal2}
\end{align}
\begin{align*}  
(\forall f,k)\textstyle\Big[(\forall l)&\textstyle(\exists N)(\forall \eps, \eps'\leq \frac{1}{N})\textstyle(\forall x \in [-\frac{1}{k},1+\frac{1}{k}])[ |\Delta_{\eps}(f(x))- \Delta_{\eps'}(f(x))|\leq\frac{1}{l}]\di \\
&\textstyle(\forall k)(\exists N)(\forall \eps> 0)\big[\eps\leq \frac{1}{N} \di | I(\Delta_{\eps}(f(\cdot)), 1))- (f(1)-f(0))|\leq\frac{1}{k}\big]\Big].
\end{align*}
\begin{align*}  
(\forall f,g,k)\textstyle\Big[(\forall l)&\textstyle(\forall \eps, \eps'\leq \frac{1}{g(l)})\textstyle(\forall x \in [-\frac{1}{k},1+\frac{1}{k}])[ |\Delta_{\eps}(f(x))- \Delta_{\eps'}(f(x))|\leq\frac{1}{l}]\di \\
&\textstyle(\forall k)(\forall \eps> 0)\big[\eps\leq \frac{1}{t(g,k)} \di | I(\Delta_{\eps}(f(\cdot)), 1))- (f(1)-f(0))|\leq\frac{1}{k}\big]\Big].
\end{align*}
The effective version is obtained from the non-effective version via the nonstandard version by applying the same steps as in the proof of Theorem \ref{floggen4}.  
Note that the notion of differentiability used in \eqref{holal2} gives rise to Bishop's definition \cite{bridge1}*{Definition~5.1, p.~44}.  
In particular, it is easy to verify that we would obtain:
\[
(\exists^{\st}g)(\forall^{\st}l)(\forall \eps, \eps' >0)\textstyle(\forall x \in [-\frac{1}{k},1+\frac{1}{k}])[ |\eps|, |\eps'| \leq \frac{1}{g(l)}\di | \Delta_{\eps}(f(x))- \Delta_{\eps'}(f(x))|\leq\frac{1}l]
\]
in the former case, i.e.\ a modulus of differentiability naturally emerges.  
\end{rem}
The inquisitive reader has no doubt noted that by the previous proof, $\M$ also proves that a \emph{pointwise} (nonstandard or $\eps$-$\delta$ relative to `st') continuous function is \emph{uniformly} nonstandard continuous, and therefore nonstandard Riemann integrable by the results in the previous section.  
In Section \ref{suckone}, we show that this fact does not contradict \cite{simpson2}*{IV.2.7}, which states that weak K\"onig's lemma (not available in $\textsf{E-PA}^{\omega}$) is equivalent to the fact that a pointwise continuous function is Riemann integrable on the unit interval.  

\subsection{General template}\label{sucktwo}
In this section, we provide a general template for extracting computational information from a proof not involving Nonstandard Analysis.  
We will limit ourselves to a theorem $T$ with a proof in $\textsf{E-PA}^{\omega}+\WKL$.  The motivation for the latter limitation is discussed at the end of this section. 
\begin{enumerate}
\item First of all, verify that $\textsf{E-PA}^{\omega}+\WKL\vdash T$ implies that $\M\vdash T$ in light of Corollary \ref{fukwkl}.  
\item Secondly, verify that the $\eps$-$\delta$-definitions (of continuity, differentiability, etc) relative to `st', are equivalent to their nonstandard versions in $\M$;  Let $T^{*}$ be $T^{\st}$ with the former replaced by the latter.  
\item Thirdly, the definitions of common notions (such as continuity, differentiability etc) in Nonstandard Analysis can be brought into the `normal form' $(\forall^{\st}x)(\exists^{\st}y)\varphi(x, y)$, where $\varphi$ is internal, i.e.\ does not involve `st'.    
This can always be done in $\M$.  
\item Fourth, the aforementioned `normal form' is closed under \emph{modes ponens}: Indeed, it is not difficult to show that an implication of the form:
\[
(\forall^{\st}x_{0})(\exists^{\st}y_{0})\varphi_{0}(x_{0}, y_{0})\di (\forall^{\st}x_{1})(\exists^{\st}y_{1})\varphi_{1}(x_{1}, y_{1}),
\]
can \emph{also} be brought into the normal form $(\forall^{\st}x)(\exists^{\st}y)\varphi(x, y)$.  
Hence, $T^{*}$ can be brought into normal form inside $\M$.  
\item Fifth, the normal form of $T^{*}$, say $(\forall^{\st}x)(\exists^{\st}y)\varphi(x, y)$, has exactly the right structure to yield the effective version of $T$, namely $(\forall x)\varphi(x, t(x))$.   
In particular, we can apply Corollary \ref{consresultcor} to the proof of the normal form $(\forall^{\st}x)(\exists^{\st}y)\varphi(x, y)$ inside $\M$.  If $x$ and $y$ are of low enough type, the majorizability predicates `$\leq^{*}$' disappear, and we obtain $(\forall x)\varphi(x, t(x))$.  
\end{enumerate}
In Section \ref{LIM}, we establish some limitations to the scope of the previous sketch.  As will become clear, we should limit ourselves to theorems provable in $\textsf{E-PA}^{\omega}+\WKL$ concerning \emph{uniform} notions (of continuity, differentiability, et cetera).  The latter limitation is also suggested by the final step.  

\medskip

In Section \ref{convieneient}, we apply the above template to theorems concerning convergence.  
In Section \ref{papprox}, we show how to extend our template to obtain approximations to objects stated to exist by (classical and ineffective) higher-order existence statements.      
Finally, in Section \ref{papadimi}, we discuss a more general framework for the second step of the above template, namely the equivalence between the `usual' and nonstandard definitions of mathematical notions, and apply it to the notion of open set and related notions.   


\section{Limitations of $\M$}\label{LIM}
In the previous section, we discovered an advantage of $\M$ over the systems from \cite{brie}, namely that $\M$ may be used to extract computational information from standard proofs, i.e.\ not involving Nonstandard Analysis.  
In this section, we discuss three \emph{disadvantages} of $\M$ compared to the systems from \cite{brie}.  
\subsection{Pointwise versus uniform definitions}\label{suckone}
In this section, we show that \emph{for our purposes} $\M$ is not suitable for studying theorems involving pointwise definitions (of e.g.\ continuity, differentiability, convergence, et cetera).  
In particular, we show that $U_{\st}$ converts the pointwise definition of (nonstandard) continuity into \emph{uniform} continuity.  
Hence, while we can prove theorems concerning pointwise continuity in $\M$, the term extraction theorem (Theorem \ref{consresult3}) converts them into theorems concerning uniform continuity anyway.   
Analogous results for differentiability and convergence are immediate.  

\medskip

Furthermore,  pointwise nonstandard continuity (and convergence, differentiability, et cetera) involves a leading standard \emph{universal} quantifier, and we shall obtain dual results for standard \emph{existential} quantifiers.  
In particular, we show that $U_{\st}$ converts the latter to the existence of rather trivial upper bounds in the case of real numbers and other higher-type objects.      
 
\medskip

First of all, we establish that $\M$ proves $\CRI'_{\ns}$, where the latter is just $\CRI_{\ns}$ from Section \ref{infidel} but with nonstandard \emph{pointwise} continuity (instead of the uniform version \eqref{soareyou4}) as follows:
\be\label{soareyou111}
(\forall^{\st} x\in [0,1])(\forall y\in [0,1])[x\approx y \di f(x)\approx f(y)].
\ee
As in the proof of Theorem \ref{floggen4}, we note that every real equals a standard real (given by a standard binary expansion).  Hence, the leading `\st' in \eqref{soareyou111} may be dropped, and $\M\vdash \CRI_{\ns}$ now yields $\M\vdash \CRI_{\ns}'$.  
Repeating the proof of Theorem~\ref{varou} for the \emph{latter} proof, we obtain that $\textup{\textsf{E-PA}}^{\omega}$ proves
\begin{align}\textstyle
\textstyle~(\forall f: \R\di \R)(\tilde{\forall} g)\Big[(\forall &\textstyle x, y \in [0,1],k)(|x-y|<\frac{1}{g(x,k)} \di |f(x)-f(y)|\leq\frac{1}{k})\label{ESTEBAN}\\
&\textstyle\di  (\forall \pi, \pi' \in P([0,1]),n)\big(\|\pi\|,\| \pi'\|< \frac{1}{t(g,n)}  \di |S_{\pi}(f)- S_{\pi}(f)|\leq \frac{1}{n} \big)  \Big],\notag
\end{align}
As always, the devil is in the detail:  We require the existence of a \emph{monotone} modulus (of pointwise continuity) $g$ in \eqref{ESTEBAN}.  
Now, the definition of $g(\cdot, k)\leq^{*}_{2}g(\cdot, k)$ for fixed $k$ is as follows:
\be\label{drak}
(\forall x^{1}, y^{1})((\forall n)x(n)\leq_{0}y(n) \di g(x, k)\leq_{0} g(y, k) ),
\ee
i.e.\ $g(\cdot, k)$ is weakly increasing for weakly increasing input.  
Now suppose the real numbers $x,y$ in \eqref{ESTEBAN} are represented by binary sequences as follows: 
\be\label{bin}
(\forall \textstyle f^{1}, h^{1},k)(|\r(\b(f))-\r(\b(h))|<\frac{1}{g(f,k)} \di |f(\r(\b(f)))-f(\r(\b(h)))|\leq\frac{1}{k}),
\ee
where $\r$ is as in the previous section and $\b(f)(k):=0$ if $f(k)=0$ and $1$ otherwise.
Then if $g$ is monotone, the associated version of \eqref{drak} yields that $g(11\dots, \cdot)$ is a modulus of \emph{uniform} continuity for $f$.  
In other words, a \emph{monotone} modulus of pointwise continuity is just 
a modulus of uniform continuity in disguise.  Hence, the interpretation 
$U_{\st}$ translates pointwise nonstandard continuity (occurring in the antecedent) into uniform continuity, and the resulting formula \eqref{ESTEBAN} is nothing more than \eqref{EST} in disguise.     

\medskip

Secondly, if pointwise (nonstandard) continuity \eqref{soareyou111} occurs positively in a theorem of $\M$, then it gets translated by $U_{\st}$ to  
\[\textstyle
(\forall k)(\tilde{\forall}x)(\forall  x', y \in [0,1])(x'\leq^{*} x\wedge |x-y|<\frac{1}{t(x,k, \dots)} \di |f(x')-f(y)|\leq\frac{1}{k}),
\]
where $t$ is a term form the internal language, depending on the other (standard) variables in the theorem at hand.  
Using binary representation as in \eqref{bin}, we note that $t$ provides a modulus of uniform continuity in exactly the same way.   
Note that the reason why we obtain \emph{uniform} definitions (of continuity, differentiability et cetera) from $U_{\st}$ and Corollary \ref{consresultcor}, is that the latter introduces 
majorizability predicates in such a way that a standard universal quantifier gives rise to $(\forall x)(\forall x'\leq^{*}x)\dots$ and the bound for the associated existential quantifier \emph{only depends on} $x$, not $x'$.  
%
%
%

\medskip

Thus, we have shown that pointwise definitions (of continuity, convergence, \dots) are translated to uniform definitions by $U_{\st}$.  
In other words, it does not seem to make much sense \emph{for our purposes} to study theorems involving these pointwise notions using $\M$.  
%
%
Now, nonstandard pointwise definitions involve a leading standard \emph{universal} quantifiers, and  the previous results `dually' imply that {standard} \emph{existential} quantifiers only result in an upper bound (in terms of majorization) when interpreted by $U_{\st}$.  
For real numbers and other higher-type objects, this upper bound provides little information, as illustrated by the following example.  
\begin{exa}\rm
Assume $\M$ can prove a version of the intermediate value theorem where $(\exists^{\st} x^{1}\in [0,1])(f(x)=0)$.  Now, applying Corollary~\ref{consresultcor} to the latter
only provides a term $t$ such that $(\exists x\in [0,1])(x\leq^{*}t(\dots) \wedge f(x)=0)$.  Moreover, if we use binary representation as follows: $(\exists^{\st} \alpha^{1})(f(\r(\b(\alpha)))=0)$, the associated upper bound for $\alpha$ can be $11\dots$ (which provides no information), or  $00\dots0011\dots$, (which only provides an approximation with fixed precision).  Note that there is an approximate version of the intermediate value theorem with arbitrary precision  (See \cite{bridge1}*{\S2.4.8}) in Constructive Analysis.  
\end{exa}
The previous example should not come as a surprise:  Kohlenbach has shown that a functional computing the intermediate value of a continuous function, also computes the Turing jump (\cite{kohlenbach2}*{Prop.\ 3.14}).   
Moreover, the interpretation $S_{\st}$ from \cite{brie} has a similar feature:  The statement $(\exists^{\st} x^{1}\in [0,1])(f(x)=0)$ would be translated into $(\exists  x^{1}\in t(\dots))(x\in [0,1]\wedge f(x)=0)$, where $t$ is a term from 
the language which provides a \emph{finite sequence} of witnesses.  Clearly, this finite sequence does not provide much information about the actual fixed point.    
%
%

\medskip

In conclusion, we should limit our scope to theorems involving \emph{uniform} definitions (of continuity, differentiability, convergence, et cetera) and avoid existence statements involving standard higher-type objects.  
Nonetheless, we establish in Section \ref{papprox} that $\M$ may be used to obtain \emph{approximations} to objects claimed to exist by existential statements, \emph{using nothing more than a classical proof of the latter}.    

\subsection{Theorems implying arithmetical comprehension}\label{limpie}
In this section, we study the suitability of $\M$ for the extraction of computational information from theorems implying (the nonstandard version of) arithmetical comprehension.    
In particular, we show that the approach from \cites{sambon,samzoo} does not work, thereby exposing another disadvantage of $\M$.  

\medskip

In more detail, explicit\footnote{An implication $(\exists \Phi)A(\Phi)\di (\exists \Psi)B(\Psi)$ is \emph{explicit} if there is a term $t$ in the language such that additionally $(\forall \Phi)[A(\Phi)\di B(t(\Phi))]$, i.e.\ $\Psi$ can be explicitly defined in terms of $\Phi$.\label{dirkske}} equivalences involving arithmetical comprehension $(\exists^{2})$  are obtained in \cite{sambon}*{\S4.1} and \cite{samzoo} as follows:  For internal $\varphi$, the equivalence 
$\paai\asa (\forall^{\st}x)(\exists^{\st}y )\varphi(x,y)$ is brought into normal form, and  term extraction is applied to obtain terms $s, u$ validating the following explicit implications:
\be\label{frigl}
(\forall \mu^{2})\big[\textsf{\MU}(\mu)\di (\forall x)\varphi(x, s(\mu)) \big] \wedge (\forall t)\big[ (\forall x)\varphi(x, t(x))\di  \MU(u(t))  \big], 
\ee
where $\MU(\mu)$ is the formula in square brackets in the following formula
\be\label{mu}\tag{$\mu^{2}$}
(\exists \mu^{2})\big[(\forall f^{1})( (\exists n)f(n)=0 \di f(\mu(f))=0)    \big],
\ee
which is also called \emph{Feferman's non-constructive search operator} and is equivalent to $(\exists^{2})$ by \cite{kohlenbach2}*{Prop.~3.9}.
By way of example, the results in \cite{samzoo} suggest that all theorems from the Reverse Mathematics zoo (\cite{damirzoo}) can be classified as in \eqref{frigl}.

\medskip

The technique sketched just now involving $\paai$ behaves as follows in $\M$.  
\begin{thm}\label{fliurk}
For internal $\varphi$, applying Corollary \ref{consresultcor} to $\M\vdash \paai \di (\forall^{\st}x^{0})(\exists^{\st}y^{0})\varphi(x,y)$ yields a term $t$ such that $\textsf{\textup{E-PA}}^{\omega}\vdash(\tilde{\forall} {\mu})(\MU(\mu)\di  (\forall x)\varphi(x, t(x,\mu)))$.
\end{thm}
\begin{proof}
Note that $\paai$ has the following normal form:
\[
(\forall^{\st}f^{1})(\exists^{\st}n)\big[(\exists m)f(m)=0\di (\exists i\leq n)f(i)=0],
\]
and hence the assumption of the theorem implies that $\M$ proves 
\[
(\forall^{\st} \mu^2, x^{0}, f^{1})(\exists^{\st}y^{0})\big[[(\exists m)f(m)=0\di (\exists i\leq \mu(f))f(i)=0]\di \varphi(x,y)\big],
\]
and applying Corollary \ref{consresultcor} finishes the proof.  
\end{proof}
By the previous theorem, $U_{\st}$ translates $\paai$ into the existence of a \emph{monotone} non-constructive search operator, \emph{which however does not exist}.  
Indeed, the proof of \cite{kohlenbach3}*{Prop.~3.70} implies that $(\mu^{2})$ is not even majorisable.  
In other words, the final formula in Theorem \ref{fliurk} is vacuously true.  

\medskip

In conclusion, we observe that the approach from \cite{sambon}*{\S4.1} does not work in $\M$.  
The same holds \emph{mutatis mutandi} for the reverse implication\footnote{Note that the term provided by Corollary \ref{consresultcor} is also monotone, which would also result in a monotone search operator in the second conjunct of \eqref{frigl}.} $\dots \di \paai$, 
and the results pertaining to $\Pi_{1}^{1}$-comprehension (and the Suslin functional) in \cite{sambon}*{\S4.5}.

\medskip

Finally, any functional computing the intermediate value (or the maximum) of a continuous function on the unit interval, also computes the Turing jump functional $(\exists^{2})$ (See \cite{kohlenbach2}*{\S3}).  
While the results in this section imply that we cannot study the former functional in $\M$, we shall obtain \emph{approximate results} in Section \ref{papprox}.

\subsection{System $\M$ and the fan functional}\label{limpie2}
In Section \ref{suckone}, we showed that \emph{for our purposes} the study of pointwise notions (of e.g.\ continuity) in $\M$ is not very useful.  
In this section, we \emph{willy-nilly} attempt this study, and in the process obtain a `no go theorem' similar to Theorem \ref{fliurk} \emph{at the level of weak K\"onig's lemma}.    

\medskip

Now, the study of pointwise notions is indeed tempting, as discussed now.
\begin{rem}[Dini's theorem and the fan functional]\rm
First of all,  \emph{Dini's theorem} is the statement 
that \emph{a monotone sequence of continuous functions converging pointwise to a continuous function on a compact space, also converges uniformly} (See e.g.\ \cite{rudin}*{Theorem 7.13, p.\ 150}).
Following \cite{diniberg}, Dini's theorem is equivalent to weak K\"onig's lemma in classical Reverse Mathematics.  Note that this lemma states the existence of non-computable objects.  

\medskip

Secondly, Dini's theorem was studied in \cite{sambon}*{\S4.3}, resulting in a term $t$ which constitutes a \emph{modulus} of uniform 
convergence, taking as input the moduli of (pointwise) continuity and (pointwise) convergence \emph{and the fan functional}.            
The latter provides a modulus of uniform continuity for (pointwise) continuous functionals on Cantor space (See \cite{noortje}*{\S4} and \eqref{frellll} below).  
Serendipitously, the fan functional has a computable representation, called a \emph{recursive {associate}} (See \cite{noortje}*{Theorem 4.2.6}), implemented in Haskell (\cite{escardooo}).       

\medskip 
    
Thirdly, all terms in G\"odel's \textsf{T} have canonical interpretations in the typed structure of the Kleene-Kreisel countable functionals (See \cite{noortje}*{\S2} for the latter) and this interpretation provides canonical {associates}. 
Thus, in Kleene's \emph{second model} (See e.g.\ \cite{beeson1}*{\S7.4, p.\ 132}), for a term $t$ of G\"odel's $\textsf{T}$ and $\Phi$ the fan functional, $t(\Phi)$ can be computed by evaluating the associate for $t$ on the associate for $\Phi$.       
\end{rem}    
The previous remark suggests that we may extract information which is \emph{computable modulo the fan functional} from theorems implying weak K\"onig's lemma.  
Although the latter states the existence of non-computable objects, there is a technical way of computing values of the fan functional, i.e.\ the obtained computational information does not `really' depend on a non-computable oracle.   

\medskip

We now show that the approach sketched in the previous remark does not work in $\M$.  To this end, recall that the fan functional $\Omega^{3}$ is defined as: 
\be\label{frellll}
(\forall \varphi^{2}\in C)\big[(\forall f,g \leq_{1}1)\big(\overline{f}\Omega(\varphi)=_{0}\overline{g}\Omega(\varphi)\di \varphi(f)=_{0}\varphi(g) \big)\big], 
\ee
where the formula in square brackets is denoted $\MUC(\Omega)$.  We also define $\MUC_{\ns}$ as
\be\label{knorii}
(\forall^{\st} \psi^{2}\in C)(\forall f, g\leq_{1}1)(f\approx_{1}g\di \psi(f)=\psi(g)).
\ee
Note that $\psi$ is assumed to be pointwise continuous.  

\begin{thm}\label{fliurk2}
For internal $\varphi$, applying Corollary \ref{consresultcor} to 
\be\label{forgoooo}
\M\vdash \big[\MUC_{\ns}\di (\forall^{\st}x^{0})(\exists^{\st}y^{0})\varphi(x,y)\big]
\ee
yields a term $t$ such that $\textsf{\textup{E-PA}}^{\omega}\vdash(\tilde{\forall} {\Omega})(\MUC(\Omega)\di  (\forall x)\varphi(x, t(x,\Omega)))$.
\end{thm}
\begin{proof}
Clearly, \eqref{knorii} implies (using classical logic) that
\[
(\forall^{\st} \psi^{2})(\forall f, g\leq_{1}1)(\exists^{\st}N^{0})(\psi\in C\wedge \overline{f}N=\overline{g}N\di \psi(f)=\psi(g)),
\]
and applying $\textsf{R}^{\omega}$ yields
\[
(\forall^{\st} \psi^{2})(\exists^{\st}N^{0})\big[(\forall f, g\leq_{1}1)(\psi\in C\wedge\overline{f}N=\overline{g}N\di \psi(f)=\psi(g))\big],   
\]
where $A(\psi, N)$ is the formula in square brackets.  
Then $\MUC_{\ns}\di (\forall^{\st}x^{0})(\exists^{\st}y^{0})\varphi(x,y)$ implies the following normal form:
\[
(\forall^{\st}\Omega^{3}, x^{0})(\exists^{\st} y^{0})[(\forall \psi^{2})A(\psi, \Omega(\psi))\di \varphi(x, y)],
\]
and the theorem follows.  
\end{proof}
By the previous theorem, $U_{\st}$ translates $\MUC_{\ns}$ into the existence of a \emph{monotone} fan functional, \emph{which however does not exist};  Indeed, as noted in \cite{ferrari1}*{p.\ 42}, the fan functional is not even majorisable, and we observe that \eqref{forgoooo} gives rise to a vacuously true statement.  
This problem remains when one replaces $\MUC_{\ns}$ in \eqref{forgoooo} by other nonstandard principles implying weak K\"onig's lemma, as we show now.  In particular, we consider the Standard part principle for Cantor space 
\be\tag{\textsf{\textup{STP}}}\label{STP}
(\forall \alpha^{1}\leq_{1}1)(\exists^{\st}\beta\leq_{1}1)(\alpha\approx_{1}\beta),
\ee
which is easily seen to imply weak K\"onig's lemma relative to `\st'.  
\begin{cor}\label{fliurk4}
For internal $\varphi$, applying Corollary \ref{consresultcor} to 
\be\label{forgoo3oo}
\M+\QFAC^{1,0}\vdash \big[\ref{STP}\di (\forall^{\st}x^{0})(\exists^{\st}y^{0})\varphi(x,y)\big]
\ee
yields a vacuous truth.  
\end{cor}
\begin{proof}
First of all, \ref{STP} is (easily seen to be) equivalent to the following: 
\begin{align}\label{WKLN}
(\forall T^{1}\leq_{1}1)\big[(\forall^{\st}\alpha^{1}\leq_{1}1)&(\exists^{\st}n^{0})(\overline{\alpha}n\not \in T)\\
&\di (\exists^{\st}k^{0})(\forall^{\st} \beta^{1}\leq_{1}1)(\exists i\leq k)(\overline{\beta}i \not\in T) \big],\notag
\end{align}
in which one immediately recognises the fan theorem, i.e.\ the classical contraposition of weak K\"onig's lemma.  Now, \eqref{WKLN} is equivalent to the normal form:  
\begin{align}\label{WKLN2}
(\forall^{\st}g^{2})(\exists^{\st}k^{0}, \alpha\leq_{1}1)\big[(\forall T^{1}\leq_{1}1)\big[&(\forall \gamma \leq^{*}_{1}\alpha)(\overline{\gamma}g(\gamma)\not \in T)\\
&\di (\forall \beta^{1}\leq_{1}1)(\exists i\leq k)(\overline{\beta}i \not\in T) \big]\big],\notag
\end{align}
where $B(g,k, \alpha)$ is the formula in outermost square brackets.  Hence, \eqref{forgoo3oo} implies 
\[
\M\vdash (\forall^{\st}x^{0}, \Xi)(\exists^{\st}y^{0})\big[(\forall g^{2})B(g, \Xi(g)(1), \Xi(g)(2)) \di \varphi(x,y)\big],
\]
and applying Corollary \ref{consresultcor} yields a term such that $\textsf{E-PA}^{\omega}$ proves 
\be\label{frikkko}
(\tilde{\forall} x^{0}, \Xi)(\forall x',\Xi'\leq^{*}x, \Xi )(\exists y^{0}\leq^{*}t(\Xi, x))\big[(\forall g^{2})B(g, \Xi(g)(1), \Xi(g)(2)) \di \varphi(x,y)\big],
\ee
Now suppose $\Xi$ is a monotone functional satisfying $(\forall g^{2})B(g, \Xi(g)(1), \Xi(g)(2))$.  If $h^{2}$ is such that $(\forall \alpha^{1}\leq_{1}1)(\overline{\alpha}h(\alpha)\not \in T)$, then the binary tree $T$ is not taller than $\Xi(h)(1)$.   
In other words, we obtain the following explicit version of the fan theorem
\begin{align}\label{WKLN3}
(\forall h^{2})\big[(\forall T^{1}\leq_{1}1)\big[&(\forall \gamma\leq_{1}1)(\overline{\gamma}h(\gamma)\not \in T)\\
&\di (\forall \beta^{1}\leq_{1}1)(\exists i\leq \Xi(h)(1))(\overline{\beta}i \not\in T) \big]\big],\notag
\end{align}
noting that $\Xi(\cdot)(1)$ is also monotone.  Clearly, \eqref{WKLN3} implies weak K\"onig's lemma assuming $\QFAC^{1,0}$.  By the results in \cite{kohlenbach4}*{\S4}, weak K\"onig's lemma implies that every $\varphi^{2}\in C$ has a so-called associate $\alpha^{1}$, i.e.\ we have
\begin{enumerate}
\item $(\forall \beta^{1})(\exists N^{0})(\alpha(\overline{\beta}N)>0)$.
\item $(\forall M, \beta)(\alpha(\overline{\beta}M)>0\di \varphi(\beta)+1=\alpha(\overline{\beta}M))$.
\end{enumerate}
The reduction in quantifier complexity afforded by associates (rather than the usual definition of continuity), allows us to define a tree as follows: $\sigma\in T_{\alpha}\asa \alpha(\sigma)=0$ where $\alpha$ is an associate for $\varphi\in C$.    
Hence, we may conclude that $\Xi$ as in \eqref{WKLN3} is `almost' the fan functional in that it works on associates of functionals (rather than the functionals themselves).  However, since every $\varphi\in C$ has an associate (thanks to weak K\"onig's lemma), the functional $\Xi$ cannot be majorisable (let alone monotone), and we obtain that \eqref{frikkko} is a vacuous truth.     
\end{proof}
%
Finally, we show that while Theorem \ref{ciro} already is in normal form, it does not yield a `supremum functional' as one would perhaps hope (or fear, as such a functional is not available in \textsf{E-PA}$^{\omega}$).  
\begin{cor}
Applying Corollary \ref{consresultcor} to Theorem~\ref{ciro}, we obtain a triviality.  
\end{cor}
\begin{proof}
X
\end{proof}
In conclusion, we have established that the approach from \cite{sambon}*{\S4} involving the fan functional (and certain variations) does not work in $\M$.  


%
%
%
%
  
\section{Proof mining using $\M$ bis}\label{BIS}
In this section, we provide more examples of how to extract computational information from proofs not involving Nonstandard Analysis.  
In particular, we show in Section \ref{papprox} that from certain \emph{classical and ineffective} existence proofs, one can `automatically' extract approximations to the objects claimed to exist.   
We show that if the object claimed to exist is unique, the approximation converges to it.  
\subsection{Theorems involving convergence}\label{convieneient}
In this section, we show how the template sketched in Section \ref{sucktwo} may be used to extract computational information from theorems involving convergence.

\medskip

First of all, we study the \emph{uniform limit theorem} (See \cite{munkies}*{Theorem 21.6}), which states that if a sequence of continuous functions \emph{uniformly} converges to another function, the latter is also continuous.  
In light of Section \ref{suckone}, we shall study the \emph{uniformly} continuous case;  We will extract the effective version from a proof not involving Nonstandard Analysis.  Nonstandard convergence is defined as follows.      
\bdefi 
A sequence $x_{n}$ \emph{nonstandard converges} to $x$ if $(\forall N\in \Omega)(x\approx x_{N})$.  
A sequence $f_{n}(x)$ \emph{nonstandard uniformly converges} to $f(x)$ on $[0,1]$ if $(\forall   x\in [0,1], N\in \Omega)(f_{N}(x)\approx f(x))$.
\edefi
Hence, we have the following nonstandard version.  
\begin{thm} For all $f_{n}, f:\R\di \R$, we have
\begin{align}
&\big[(\forall^{\st} n^{0})(\forall x^{1}, y^{1}\in [0,1])[x\approx y \di f_{n}(x)\approx f_{n}(y)] ~\wedge\label{from2}\tag{\textsf{ULC}$_{\ns}$} \\
& (\forall x\in [0,1])(\forall N\in \Omega)( f_{N}(x)\approx f(x)) \big]
\di (\forall x\in [0,1])(\forall x,y\in [0,1])(x\approx y\di f(x)\approx f(y)).\notag
\end{align}
\end{thm}
Let $\ULC$ denote the statement that a sequence of uniformly continuous functions {uniformly} converges to another function, the latter is also uniformly continuous.
Let $C(f, g, S)$ be the statement that $f$ is uniformly continuous on $S$ with modulus $g$ and let $D(f_{n}, f,h, S)$ be the statement that $h$ is a modulus of \emph{uniform} convergence for $f_{n}\di f$ on $S$, i.e.\ $h$ does not depend on the choice of point in $S$.  
Let $\ULC_{\ef}(t)$ denote the statement that for all $ f, f_{n}:\R\di \R$, $g_{n}^{2}$, and $h^{1}$ we have
\be\label{kirf}
\big[(\forall n^{0})C(f_{n},g_{n}, [0,1])\wedge D(f_{n}, f, h, [0,1])   \big] \di C(f, t( g_{n}, h), [0,1]).
\ee
We have the following theorem.
\begin{thm}\label{floggen}
From a proof of $\ULC$ in $\textsf{\textup{E-PA}}^{\omega}$, a term $t$ can be extracted such that $\textup{\textsf{E-PA}}^{\omega} $ proves $\ULC_{{\ef}}(t)$.  
\end{thm}
\begin{proof}
First of all, it is straightforward to prove $\ULC$ inside $\textup{\textsf{E-PA}}^{\omega}$, yielding a proof of $\ULC^{\st}$ in $\M$.
To obtain $\ULC_{\ns}$ from this standard statement, strengthen the antecedent to \emph{nonstandard} uniform continuity and convergence (which is trivial) and strengthen the consequent to nonstandard uniform continuity (in the same way as for Theorem~\ref{floggen4}).  

\medskip

Secondly, for completeness, we directly prove $\ULC_{\ns}$ in $\M$.  To this end, let $f$ and $f_{n}$ be as in the antecedent of $\ULC_{\ns}$.  
Now fix standard  $x, _{0}y_{0}\in [0,1]$ such that $x_{0}\approx y_{0}$.  By assumption, we have $(\forall^{\st}n)(f_{n}(x_{0})\approx f_{n}(y_{0}))$, implying  $(\forall^{\st}n, k)(|f_{n}(x_{0})- f_{n}(y_{0})|\leq \frac{1}{k})$.  
By overflow (See Section \ref{NSANSA}), there is $N_{0}\in \Omega$ such that $(\forall  n, k\leq N_{0})(|f_{n}(x_{0})- f_{n}(y_{0})|\leq \frac{1}{k})$.  
Again by assumption, we have $f(x_{0})\approx f_{N_{0}}(x_{0})\approx f_{N_{0}}(y_{0})\approx f(y_{0})$, which implies the nonstandard uniform continuity of $f$.       
Alternatively, $\ULC_{\ns}$ can be proved in $\EPA_{\st}^{\omega}$ by proving overflow using induction rather than idealisation.  

\medskip

Thirdly, we show that $\ULC_{\ns}$ can be brought into the right normal form for applying Corollary \ref{consresultcor}.  
Making explicit all standard quantifiers in the first conjunct in the antecedent of $\ULC_{\ns}$, we obtain: 
\[\textstyle
(\forall^{\st} n^{0})(\forall x^{1}, y^{1}\in [0,1])[(\forall^{\st}N)|x-y|\leq \frac{1}{N} \di (\forall^{\st}k)|f_{n}(x)- f_{n}(y)|\leq \frac{1}{k}]
\]
Bringing the standard quantifiers to the front as much possible, this becomes
\[\textstyle
(\forall^{\st}n^{0}, k^{0})(\forall x^{1}, y^{1}\in [0,1])(\exists^{\st}N)[|x-y|\leq \frac{1}{N} \di |f_{n}(x)- f_{n}(y)|\leq \frac{1}{k}].
\]
Since the formula in square brackets is internal, we may apply $\textsf{R}^{\omega}$ to obtain
\[\textstyle
(\forall^{\st} n^{0}, k^{0})(\exists^{\st} m^{0})(\forall x^{1}, y^{1}\in [0,1])(\exists N\leq m)[|x-y|\leq \frac{1}{N} \di |f_{n}(x)- f_{n}(y)|\leq \frac{1}{k}], 
\]
which immediately yields 
\[\textstyle
(\forall^{\st} n^{0}, k^{0})(\exists^{\st} m^{0})(\forall x^{1}, y^{1}\in [0,1])[|x-y|\leq \frac{1}{N} \di |f_{n}(x)- f_{n}(y)|\leq \frac{1}{k}], 
\]
Now apply $\MAC$ to obtain 
\[\textstyle
(\tilde{\exists}^{\st} g_{(\cdot)})(\forall^{\st} n^{0}, k^{0})\big[(\forall x^{1},y^{1}\in [0,1])[|x- y|\leq\frac{1}{g_{n}(k)} \di | f_{n}(x)- f_{n}(y)| \leq \frac{1}{k}]\big].
\]
For brevity let $A(\cdot)$ be the formula in big square brackets.  Similarly, the second conjunct of the antecedent of $\ULC_{\ns}$ yields
\[\textstyle
(\tilde{\exists}^{\st} h)(\forall^{\st}k)\big[(\forall x\in [0,1])(\forall N\geq h(x, k))(| f_{N}(x)- f(x)|\leq \frac{1}{k})\big],    
\]
where $B(\cdot)$ is the formula in big square brackets.  Lastly, the consequent of $\ULC_{\ns}$, which is just the nonstandard uniform continuity of $f$, implies 
\[\textstyle
(\forall^{\st} k)(\exists^{\st}N)\big[(\forall x^{1}, y \in [0,1])(|x-y|<\frac{1}{N} \di |f(x)-f(y)|<\frac{1}{k})\big],
\]
where $E(\cdot)$ is the formula in big square brackets.  The monotonicity of $g, h$ is again not relevant, hence we drop this condition.  We may also drop the `st' predicates on the universal quantifiers \emph{in the antecedent}.      
Hence, $\ULC_{\ns}$ implies that 
\begin{align*}
(\forall f, f_{n})\Big[\big[(\exists^{\st} g_{(\cdot)})( \forall n^{0}, k^{0})A(f_{n}, g_{n}, n, k)& \wedge ({\exists}^{\st} h)(\forall k')B(f, f_{n}, h, k')  \big]\\
&\di (\forall^{\st} k'' )(\exists^{\st}N)E(f, N , k'')\Big].
\end{align*}
Bringing all standard quantifiers to the front, we obtain
\begin{align}
(\forall^{\st} g_{(\cdot)}, h&,k'')\underline{(\forall f, f_{n})(\exists^{\st}  N})\label{fruk}\\
&\Big[\big[(\forall n, k)A(f_{n}, g_{n}, x, n, k) \wedge  (\forall k')B(f, f_{n}, h, k')  \big]\di E(f, N , k'', x')\Big],\notag
\end{align}
where the formula in big(gest) square brackets is internal.  Applying $\textsf{R}^{\omega}$ to the underlined quantifier alternation in \eqref{fruk}, we obtain 
\begin{align}
(\forall^{\st} g_{(\cdot)}, h&,k'')(\exists^{\st}N'){(\forall f, f_{n})(\exists N\leq^{*}_{0}N')}\label{fruk2}\\
&\Big[\big[(\forall n, k)A(f_{n}, g_{n}, x, n, k) \wedge  (\forall k')B(f, f_{n}, h, k')  \big]\di E(f, N , k'', x')\Big],\notag
\end{align}
As \eqref{fruk2} was proved in $\M$, Corollary \ref{consresultcor} tells us that $\textup{\textsf{E-PA}}^{\omega}$ proves 
\begin{align}
(\tilde{\forall} g_{(\cdot)}, h&,k''){(\forall f, f_{n})(\exists N\leq^{*}t(g_{(\cdot)}, h, k''))}\label{frakker}\\
&\Big[\big[(\forall n, k)A(f_{n}, g_{n}, x, n, k) \wedge  (\forall k')B(f, f_{n}, h, k')  \big]\di E(f, N , k'', x')\Big],\notag
\end{align}
where $t$ is a term in the language of $\textup{\textsf{E-PA}}^{\omega}$.  
Define $s(g_{(\cdot)}, h, k''):=t(\widetilde{g_{(\cdot)}}, \tilde{h}, k'')$ as in \eqref{crux} and note that \eqref{frakker} yields $\ULC_{\ef}(s)$.  
\end{proof}
There is a constructive version of the uniform limit theorem in Bishop's  Constructive Analysis (See \cite{bridge1}*{Prop.\ 1.12, p.\ 86}).  

\medskip

The uniform limit theorem serves as an example of the class of theorems (involving convergence in the antecedent) which can be treated using the template described in Section \ref{sucktwo}.  
We now consider  the following more challenging theorem, which is \cite{rudin}*{Theorem 7.17, p.\ 152} and moreover contains convergence \emph{in the consequent}.
\begin{thm}\label{fryg}
Suppose ${f_n}$ is a sequence of functions, differentiable on  $[a, b]$, and such that ${f_n(x_0)}$ converges for some point $x_0$ on $[a, b]$. If  $f'_n$ converges uniformly on $[a, b]$, then ${f_n}$ converges uniformly to a function $f$, and  $f'(x) = \lim_{n\to \infty} f'_n(x)$ for $ x \in [a, b]$.
\end{thm}
%
We now show that the usual $\eps$-$\delta$-definition of convergence (relative to `st') is equivalent to the nonstandard one in $\M$.  Note that we use the `real' version of convergence related to `$\lim_{\eps\di 0}$' instead of `$\lim_{n\di \infty}$'.  
\begin{thm}\label{un2ward}
The system $\M$ proves that standard and nonstandard convergence are equivalent, i.e.\ for every $\R\di \R$-function $x_{(\cdot)}^{1\di 1}$ and real $x$, we have
\begin{align}\label{unlikely}\textstyle
&(\forall \eps^{1}\ne0)(\eps\approx 0\di x_{\eps}\approx x) \\
&\asa\textstyle (\forall^{\st}k^{0})(\exists^{\st}N^{0})(\forall^{\st} \eps^{1}\ne0)(|\eps|\leq \frac{1}{N}\di    |x_{\eps}-x|\ll \frac{1}{k}).\notag
\end{align}
\end{thm} 
\begin{proof}
For the forward implication in \eqref{unlikely}, nonstandard convergence immediately yields $(\forall^{\st}k)(\forall M\in \Omega)(\forall \eps\ne0)  (|\eps|\leq\frac{1}{M}\di  |x_{\eps}- x|\leq \frac{1}{k})$, which in turn yields  
\[\textstyle
(\forall^{\st}k)(\exists N)(\forall M\geq N)(\forall \eps\ne0)  (|\eps|\leq\frac{1}{M}\di  |x_{\eps}- x|\leq \frac{1}{k})),
\]
and using minimisation\footnote{We could also use \emph{underspill} instead (See \cite{brie}*{\S3}) of minimisation, so as to guarantee that the proof goes through in a system not stronger than $\PRA$.}, there is a \emph{least} such $N$.  By assumption, this $N$ cannot be infinite, and is therefore finite, i.e.\ 
\be\label{vortek}\textstyle
(\forall^{\st}k)(\exists^{\st} N)(\forall M\geq N)(\forall \eps\ne0)  (|\eps|\leq\frac{1}{M}\di  |x_{\eps}- x|\leq \frac{1}{k}),
\ee
and we obtain the right-hand side of \eqref{unlikely}, even with the right-most `$\st$' dropped.  

\medskip

For the reverse implication in \eqref{unlikely}, the right-hand side of the latter implies 
\be\label{trotec}\textstyle
 (\forall^{\st}k^{0})(\exists^{\st}N^{0})(\forall^{\st} \alpha^{1}\leq_{1}1 )(0<|\r(\alpha)|\leq \frac{1}{N}\di    |x_{\r(\alpha)}-x|\ll \frac{1}{k})
\ee
By Theorem \ref{OMG}, every binary sequence is standard in $\M$, and \eqref{trotec} becomes 
\be\label{trotec2}\textstyle
 (\forall^{\st}k^{0})(\exists^{\st}N^{0})(\forall \alpha^{1}\leq_{1}1 )(0<|\r(\alpha)|\leq \frac{1}{N}\di    |x_{\r(\alpha)}-x|\ll \frac{1}{k}).
\ee
However, every (standard or nonstandard) real $\xi^{1}$ also has a \emph{standard} binary expansion $\alpha^{1}$ such that $\xi=_{\R}\r(\alpha)$, yielding $x_{\xi}=_{\R}x_{\r(\alpha)}$ by \eqref{REXT}. 
Hence, \eqref{trotec2} implies
\[
\textstyle (\forall^{\st}k^{0})(\exists^{\st}N^{0})(\forall \xi^{1}\ne0)(|\eps|\leq \frac{1}{N}\di    |x_{\xi}-x|\ll \frac{1}{k}),
\]
which immediately yields $x_{\eps}\approx x$ for $0\ne \eps\approx 0$.  
\end{proof} 
Let $\textsf{ULD}$ be Theorem \ref{fryg} formulated with the uniform version of differentiability \emph{and the notion of limit} `$\lim_{\eps\di 0}$' instead of `$\lim_{n\di \infty}$'; Let $\ULD_{\ef}(t)$ and $\ULD_{\ns}$ be the obvious effective and nonstandard versions.   
In light of Theorem \ref{un2ward}, a proof of $\ULD$ \emph{not involving Nonstandard Analysis} gives rise to a proof of $\ULD_{\ns}$.  
The following corollary is thus immediate.  
\begin{cor}  
  From a proof of $\ULD$ in $\textsf{\textup{E-PA}}^{\omega}$, a term $t$ can be extracted such that $\textup{\textsf{E-PA}}^{\omega} $ proves $\ULD_{{\ef}}(t)$.  
\end{cor}  
\begin{proof}
Follow the template sketched in Section \ref{sucktwo}.  Note that the normal form of convergence is given by \eqref{vortek}.  
\end{proof}
In the previous corollary, we made use of the notion of convergence related to `$\lim_{\eps\di 0}$', as we do not know how to prove a version of \eqref{unlikely} for the notion of convergence related to `$\lim_{n\di \infty}$'.  
%
%
%
%
%

%
%
%
%
%
%
%

\subsection{Approximations via $\M$}\label{papprox}
In this section, we show that from certain \emph{classical and ineffective} existence proofs, we can `automatically' extract approximations to the objects claimed to exist, thanks to $\M$.   
We also show that if the object claimed to exist is unique, the approximation converges to it.  Note that Wattenberg has hinted at the possibility of this kind of results in \cite{watje}.      

\medskip

Examples of theorems which can be treated in this way are:  The intermediate value theorem, the Weierstra\ss~maximum theorem, the Peano existence theorem, and the Brouwer fixed point theorem (See \cite{simpson2}*{II.6 and IV.2}).  

\medskip

First of all, we treat the Weierstra\ss~maximum theorem, which is the statement that every continuous function on the unit interval attains its maximum.  In light of Section \ref{suckone}, we must limit ourselves to uniformly continuous functions.  
Thus, let $\WEI$ be the statement that for all $f$ such that \eqref{soareyou6}, we have $(\exists x \in [0,1])(\forall y\in [0,1])(f(y)\leq f(x))$.   
Note that $\WEI$ is still equivalent to weak K\"onig's lemma in Reverse Mathematics by \cite{simpson2}*{IV.2.3}.  Furthermore, a functional computing the maximum from $\WEI$, computes the Turing jump by \cite{kohlenbach2}*{Prop.\ 3.14}.  
In other words, we cannot do much better than the following \emph{approximate} version \emph{from the computable point of view}.  
\begin{align}\textstyle
\textstyle~(\forall f: \R\di \R, g)\Big[(\forall &\textstyle x, y \in [0,1],k)(|x-y|<\frac{1}{g(k)} \di |f(x)-f(y)|\leq\frac{1}{k})\label{PEST}\notag\\
&\textstyle\di  (\forall k)(\forall y\in [0,1])\big( f(y)\leq f(t(g, k))+\frac{1}{k}    \big)\Big].\tag{$\WEI^{\app}(t)$}
\end{align}
\begin{thm}\label{kiol}  
From a proof of $\WEI$ in $\textsf{\textup{E-PA}}^{\omega}+\WKL$, a term $t$ can be extracted such that $\textup{\textsf{E-PA}}^{\omega} $ proves $\WEI^{\app}(t)$.  
\end{thm}  
\begin{proof}
Clearly, a proof of $\WEI$ in $\textsf{E-PA}^{\omega}+\WKL$ yields a proof of $\WEI^{\st}$ in $\M$ as the latter proves $\WKL^{\st}$ by Corollary \ref{fukwkl}. 
Next, we show that $\WEI^{\st}$ implies the following nonstandard version of $\WEI$ in $\M$:
\begin{align}
(\forall f:\R\di \R)\big[(\forall x, y\in [0,1])&[x\approx y \di f(x)\approx f(y)] \label{PST}\\
&\di \textstyle (\forall^{\st}k^{0})(\exists^{\st}q^{0}\in [0,1])(\forall y^{1}\in [0,1]) (f(y)\leq f(q)+\frac{1}{k}) \big],\notag
\end{align}
where it is implicitly assumed that $q^{0}$ codes a rational number.  As in the proof of Theorem \ref{floggen4}, uniform $\eps$-$\delta$-continuity relative to `st' (as in the antecedent of $\WEI^{\st}$) is equivalent 
to nonstandard uniform continuity.  On the other hand, the consequent of $\WEI^{\st}$, namely  $(\exists^{\st}x^{1}\in [0,1])(\forall^{\st} y^{1}\in [0,1]) (f(y)\lessapprox f(x))$, immediately yields    
\[\textstyle
(\forall^{\st}k)(\exists^{\st}x^{1}\in [0,1])(\forall^{\st} y^{1}\in [0,1]) (f(y)\leq f(x)+\frac{1}{k}),
\]
and continuity (of the $\eps$-$\delta$-variety relative to `st') implies that
\[\textstyle
(\forall^{\st}k)(\exists^{\st}q^{0}\in [0,1])(\forall^{\st} y^{1}\in [0,1]) (f(y)\leq f(q)+\frac{1}{k}),
\]
and we trivially obtain that
\[\textstyle
(\forall^{\st}k)(\exists^{\st}q^{0}\in [0,1])(\forall^{\st} \alpha^{1}\leq_{1}1) (f(\r(\alpha))\leq f(q)+\frac{1}{k}).
\]
Since all binary sequences are standard in $\M$ by Theorem \ref{OMG}, for every real $y\in [0,1]$, there is a standard $\alpha^{1}\leq_{1}1$ such that $y=_{\R}\r(\alpha)$.  
By the definition of $\R\di \R$-function, we have $f(\r(\alpha))=_{\R}f(y)\leq f(q)+\frac{1}{k}$.    
Hence, we obtain
\be\label{koss}\textstyle
(\forall^{\st}k)(\exists^{\st}q^{0}\in [0,1])(\forall y^{1}\in [0,1]) (f(y)\leq f(q)+\frac{1}{k}),
\ee
and \eqref{PST} is seen to follow from $\WEI^{\st}$.  Clearly, the consequent of \eqref{PST} is in normal form, and the entire implication can be brought into normal form in the same way as for $\CRI_{\ns}$ in the proof of Theorem \ref{varou}.  
To make sure a useful term is extracted, modify \eqref{koss} to the following equivalent normal form
\be\label{doddfrank}\textstyle
(\forall^{\st}k)(\exists^{\st}q^{0}\in [0,1])(\exists^{\st}l^{0})(\forall y^{1}\in [0,1]) \big(\frac{1}{2^{l}}<[f(q)+\frac{1}{k}-f(y)](l)\big),
\ee
Applying term extraction (via Corollary~\ref{consresultcor}) to \eqref{PST} with this modified consequent, yields a term $t$ providing an upper bound $n_{0}$ on the code for $q^{0}$ and $l^{0}$.  
Let $s$ output that $q^{0}$ with code at most $n_{0}$ such that $[f(q)](n_{0})$ is maximal.      
Then $\textsf{E-PA}^{\omega}$ proves $\WEI^{\app}(s)$ and we are done.  
\end{proof}
Let $\WEI(t)!$ be $\WEI$ with the condition that $f$ has \emph{at most one maximum}, i.e.\ 
\[\textstyle
(\forall x, y\in [0,1])(x\ne y\di f(x)<\sup_{x\in [0,1]} f(x) \vee f(y)<\sup_{x\in [0,1]} f(x) ),
\]
and the extra conclusion that the maximum of $f$  is given by $x_{0}=(t(g, 2^{g(\cdot)}))$.  
\begin{cor}\label{feabba}
The system $\textup{\textsf{E-PA}}^{\omega}+\WKL $ proves $ \WEI(t)!$, where $t$ is the term from the theorem.
\end{cor}
Note that unique existence theorems have also been considered in the context of constructive Reverse Mathematics (See \cite{ishberg}).  

\medskip

We emphasise that the results in the previous theorem and corollary are not necessarily surprising or deep in and of themselves, especially in light of known results in proof mining such as \cites{kohlenbach3, earlyk, koliva}.  
What \emph{is} in our opinion interesting and surprising is that (i) the original proof does not involve Nonstandard Analysis and is ineffective, and (ii) the proof of Theorem~\ref{kiol} provides a template which can be applied to existence statements provable in $\WKL_{0}$.  
For instance, the previous proof immediately generalises to the theorems in the following example.  
\begin{exa}[Approximation theorems]\rm ~
\begin{enumerate}
\item The intermediate value theorem.  The latter relative to `st' has consequent $(\exists^{\st}x)(f(x)\approx 0)$, yielding the normal form $(\forall^{\st}k^{0})(\exists^{\st}q^{0})(|f(q)|\leq \frac{1}{k})$.  
\item  The Peano existence theorem.  The latter relative to `st' has consequent $(\exists^{\st}\phi^{1\di 1})(\forall^{\st}x\in [0,1])(\phi'(x)\approx f(x,\phi(x)))$, yielding $(\forall^{\st}k)(\exists^{\st}P^{0})(\forall x\in [0,1])(|P'(x)- f(x,P(x))|\leq \frac1k)$, where $P$ is a polynomial.  
\item The Brouwer fixed point theorem. The latter relative to `st' has consequent $(\exists^{\st}x)(f(x)\approx x)$, yielding the normal form $(\forall^{\st}k)(\exists^{\st}q^{0})(|f(q)-q|\leq \frac{1}{k})$.
\end{enumerate}
\end{exa}
As show in \cite{kohlenbach2}*{\S3}, a functional computing any of the objects from these three theorems, also computes the Turing jump functional, i.e.\ approximations are the best we can do in general.  
However, in the special case the object claimed to exist is \emph{unique} (like in Corollary \ref{feabba} or the supremum of a continuous function), our approximations do converge to the unique solution.    

\medskip

In general, the so-called Big Five category of theorems equivalent to $\WKL_{0}$ is perhaps the largest collection of mathematical theorems in Reverse Mathematics.  
Furthermore, the Reverse Mathematics of $\WWKL_{0}$, the latter a subsystem of $\WKL_{0}$, includes the development of measure theory (See \cite{simpson2}*{X.1}).  
In other words, the template provided by the proof of Theorem~\ref{kiol} is quite general.  

\medskip

So far, we have not discussed the complexity of the terms $t$ extracted from proofs in $\M$.  
For instance, in Theorem \ref{kiol}, the `input' proof does not take place in Nonstandard Analysis, and it is a natural question if we can obtain `more efficient' terms $t$ from proofs in $\M$ which do explicitly exploit techniques from Nonstandard Analysis.
Now, more efficient terms are usually based on short(er) proofs in proof mining (See e.g.\ \cite{kohlenbach3}*{\S12}), and we provide an example of a `very short' proof in $\M$.  
\begin{rem}[Short proofs]\rm
The following version of \eqref{PST} has a trivial proof:
\begin{align*}
(\forall^{\st}g^{1})(\forall f:\R\di \R)&\textstyle\big[(\forall x, y\in [0,1],l^{0})(|x-y| \leq\frac{1}{g(l)}\di |f(x)-f(y)|\leq\frac{1}{l} ) \\
&\di \textstyle (\forall^{\st}k^{0})(\forall^{\st} y\in [0,1])(\exists^{\st}q^{0}\in [0,1]) (f(y)\ll f(q)+\frac{1}{k+1}) \big],
\end{align*}  
where we again assume that $q^{0}$ is coded by a natural number.  
Perform the usual replacement of reals by their binary representation for the variable $y$, invoke the standardness of the former via Theorem \ref{OMG}, and conclude that `\st' may be dropped in the quantifier $(\forall^{\st}y\in [0,1])$ thanks to \eqref{REXT}.  
We thus obtain the following:
\begin{align}
(\forall^{\st}g^{1})(\forall f:\R\di \R)&\textstyle\big[(\forall x, y\in [0,1], l^{0})(|x-y| \leq\frac{1}{g(l)}\di |f(x)-f(y)|\leq\frac{1}{l} )  \label{YST}\\
&\di \textstyle (\forall^{\st}k^{0})(\forall y\in [0,1])(\exists^{\st}q^{0}\in [0,1]) (f(y)\ll f(q)+\frac{1}{k+1}) \big],\notag
\end{align}  
and applying $\textsf{R}^{\omega}$ to the consequent yields:
\begin{align}
(\forall^{\st}g^{1})(\forall f:\R\di \R)&\textstyle\big[(\forall x, y\in [0,1], l^{0})(|x-y| \leq\frac{1}{g(l)}\di |f(x)-f(y)|\leq\frac{1}{l} )  \label{XST}\\
&\di \textstyle (\forall^{\st}k^{0})(\exists^{\st}q^{0}\in [0,1])(\forall y\in [0,1])(\exists r\leq_{0}^{*}q) (f(y)<f(r)+\frac{1}{k+1}) \big].\notag
\end{align}  
In the consequent of \eqref{XST}, `$\leq_{0}^{*}$' is just `$\leq_{0}$', and note that there is $r_{0}\leq_{0} q$ such hat $f(r_{0})$ is maximal.  Hence, \eqref{XST} yields:
\begin{align}
(\forall^{\st}g^{1})(\forall f:\R\di \R)&\textstyle\big[(\forall x, y\in [0,1], l^{0})(|x-y| \leq\frac{1}{g(l)}\di |f(x)-f(y)|\leq\frac{1}{l} )  \label{XST2}\\
&\di \textstyle (\forall^{\st}k^{0})(\exists^{\st}q^{0}\in [0,1])(\forall y\in [0,1]) (f(y)<f(q)+\frac{1}{k+1}) \big].\notag
\end{align}  
Bringing all standard quantifiers outside (using again $\textsf{R}^{\omega}$), we obtain that $\M$ proves the normal form of \eqref{XST2}.    
Applying Corollary \ref{consresultcor} to the modification involving \eqref{doddfrank}, now yields the same conclusion as in the theorem.   
\end{rem}

\subsection{Equivalence between standard and nonstandard definitions}\label{papadimi}
In this section, we discuss the second step in the template from Section \ref{sucktwo}, namely the equivalence between standard and nonstandard definitions in $\M$.  
We shall obtain a general criterion for when such an equivalence holds, and apply this criterion to the definition of open sets.  
\subsubsection{A general criterion}
The second step in the template from Section \ref{sucktwo} relies on the fact that the `$\eps$-$\delta$' definition relative to `st' (of continuity, differentiability, et cetera) and the latter's nonstandard version are equivalent in $\M$.  
In this context, the following property \eqref{rextje} is important:   
\be\label{rextje}\tag{\textsf{REF}}
(\forall z^{1}, w^{1})\big[z=_{\R}w\di [\psi(z) \asa \psi(w)]  \big].
\ee
We say that $\psi$ is \emph{real extensional} if it satisfies \eqref{rextje}, and it is easy to find formulas concerning real numbers which do not satisfy \eqref{rextje}.  
However,  the innermost universal formula $(\forall x^{1}, y^{1})\varphi(x,y, k, N, f)$ of $\eps$-$\delta$-continuity as in \eqref{soareyou6} (and similarly for \eqref{pothead}) has the property \eqref{rextje}, if we suppress the parameters $f, k,N$ and code $x,y$ as one real $z$, converting $\varphi(x,y, k,N,f)$ into $\psi(z)$. 
Clearly, \eqref{rextje} is the crucial step in the proofs of Theorem \ref{floggen4} and \ref{un2ward} to guarantee the equivalence between the nonstandard and usual definitions (of continuity, differentiability, and convergence).  

\medskip

In general, one proves, in the same way as in the proof of Theorem \ref{floggen4}, that $(\forall^{\st}z^{1})\xi(z)\di (\forall z)\xi(z)$ in $\M$ for real extensional formulas $\xi$.  
Furthermore, the property \eqref{rextje} implies the following:
\be\label{rextje2}\textstyle
(\forall z^{1}, w^{1})(\exists n^{0})\big(\big[ \psi(z) \wedge \big|[z](n)-[w](n)\big|\leq \frac{1}{2^{n}} \big]\di  \psi(w)  \big),
\ee
which can intuitively be described as `the truth of $\psi(z)$ is already determined by a finite approximation of the real $z$'.  As it happens, a similar `finite approximation' condition holds 
for choice sequences in intuitionistic mathematics, namely the \emph{axiom of open data} (See e.g. \cite{troelstrac}*{\S2.2.6}).  

\medskip

In conclusion, the property \eqref{rextje} plays a central role related to the equivalences between standard and nonstandard definitions in $\M$;  This property also has 
a natural interpretation as in \eqref{rextje2}.     
\subsubsection{A case study: Open sets}
By way of example, we now establish that for open sets of real numbers, the usual definition and the nonstandard one (See \cite{loeb1}*{p.\ 75}) also coincide inside $\M$, thanks to \eqref{rextje}.  

\medskip

Analogous to Definition \ref{keepintireal}, a set of reals $U\subseteq \R$ is given by a functional $f_{U}^{2}:\R\di \N$ where `$x^{1}\in U$' is short for $f_{U}(x)=1$.  
Note that with this definition, the formula `$x\in U$' satisfies \eqref{rextje}, as $f_{U}$ is real extensional as in \eqref{REXT} from Definition~\ref{keepintireal}.  
The following theorem also goes through for the Reverse Mathematics definition of open set (See \cite{simpson2}*{II.5.6}).   
\begin{thm}[$\M$]
The definition of open set relative to `\st' is equivalent to the nonstandard version, i.e.\ for every $U^{2}$, we have
\begin{align}
&(\forall^{\st}x^{1}\in U)(\exists^{\st} r^{1}\gg0)(\forall^{\st}y^{1})(|x-y|\ll r\di y\in U)\notag \\
& ~\asa (\forall^{\st}x^{1} \in U)(\forall y^{1})(x\approx y \di y\in U).\label{fagotmuziekinstrument}
\end{align}
\end{thm}
\begin{proof}
Let $U$ be as in the right-hand side of \eqref{fagotmuziekinstrument} and resolve `$\approx$' to obtain:
\[\textstyle
(\forall^{\st}x^{1} \in U)(\forall y^{1})(\exists^{\st}N^{0})(|x-y|<\frac{1}{N} \di y\in U).  
\]
Now apply $\textsf{R}^{\omega}$ to obtain
\[\textstyle
(\forall^{\st}x^{1} \in U)(\exists^{\st}M^{0})(\forall y^{1})(\exists N^{0}\leq M)(|x-y|<\frac{1}{N} \di y\in U), 
\]
which immediately implies the left-hand side of \eqref{fagotmuziekinstrument}.  For the forward implication, assume the left-hand side of \eqref{fagotmuziekinstrument}, which implies     
\be\label{sipro}\textstyle
(\forall^{\st}x^{1}\in U)(\exists^{\st} n^{0})(\forall^{\st}y^{1})(|x-y|<\frac{1}{n+1}\di y\in U)
\ee
Now fix standard $x_{0}\in U$ and let $n_{0}$ be the associated number as in \eqref{sipro}.  Then for any $y>x_{0}$ (The case $y<x_{0}$ is treated analogously) such that $|x_{0}-y|<\frac{1}{n_{0}+1}$, there is $z_{0}^{1}\in [0,1]$ such that $x_{0}=_{\R}y-z_{0}$, and this $z_{0}$ has a (necessarily) standard binary expansion $\alpha^{1}_{0}\leq_{1}1$, i.e.\ $z_{0}=_{\R}\r(\alpha_{0})$, implying that $y=_{\R}x_{0}-\r(\alpha_{0})$.  Since $x_{0}-\r(\alpha_{0})$ is standard, we have $x_{0}-\r(\alpha_{0})\in U$, implying that $y\in U$ by \eqref{rextje}, as required to yield the right-hand side of \eqref{fagotmuziekinstrument}.  
\end{proof}
The generalisation of the previous theorem to more general (topological) spaces is now straightforward, especially in light of \cite{loeb1}*{Definition~3.1.6}.  
Furthermore, since fundamental mathematical notions like connectedness, compactness, and continuity can be defined in terms of open sets, it is straightforward to prove that the aforementioned notions immediately imply the associated nonstandard definitions, when working in $\M$ (as we established above for continuity).  In turn, theorems dealing with these notions can be treated using the template 
from Section \ref{sucktwo}.

\begin{ack}\rm
This research was supported by the following funding bodies: FWO Flanders, the John Templeton Foundation, the Alexander von Humboldt Foundation, and the Japan Society for the Promotion of Science.  
The author expresses his gratitude towards these institutions. 
The author would like to thank Ulrich Kohlenbach and Helmut Schwichtenberg for their valuable advice.  
\end{ack}

\bye